\pdfoutput=1
\RequirePackage{ifpdf}
\ifpdf % We are running pdfTeX in pdf mode
\documentclass[pdftex]{sigma}
\else
\documentclass{sigma}
\fi

\usepackage{mathtools,abraces,stmaryrd}

\numberwithin{equation}{section}

\newtheorem{Theorem}{Theorem}[section]
\newtheorem{Corollary}[Theorem]{Corollary}
\newtheorem{Conjecture}[Theorem]{Conjecture}
\newtheorem{Lemma}[Theorem]{Lemma}
\newtheorem{Proposition}[Theorem]{Proposition}
 { \theoremstyle{definition}
\newtheorem{Definition}[Theorem]{Definition}
\newtheorem{Example}[Theorem]{Example}
\newtheorem{Remark}[Theorem]{Remark}
\newtheorem{Question}[Theorem]{Question} }

\newcommand{\mb}[1]{\mathbb{#1}}

\newcommand{\mc}[1]{\mathcal{#1}}

\newcommand{\Z}{\mb{Z}}

\newcommand{\M}{\mc{M}}
\newcommand{\Mbar}{\overline{\M}}
\newcommand{\Mgn}{\Mbar_{g,n}}

\begin{document}
\allowdisplaybreaks

\newcommand{\arXivNumber}{1902.02742}

\renewcommand{\thefootnote}{}

\renewcommand{\PaperNumber}{080}

\FirstPageHeading

\ShortArticleName{Half-Spin Tautological Relations and Faber's Proportionalities of Kappa Classes}

\ArticleName{Half-Spin Tautological Relations\\ and Faber's Proportionalities of Kappa Classes\footnote{This paper is a~contribution to the Special Issue on Integrability, Geometry, Moduli in honor of Motohico Mulase for his 65th birthday. The full collection is available at \href{https://www.emis.de/journals/SIGMA/Mulase.html}{https://www.emis.de/journals/SIGMA/Mulase.html}}}

\Author{Elba~GARCIA-FAILDE~$^\dag$, Reinier KRAMER~$^\ddag$, Danilo LEWA\'{N}SKI~$^\ddag$ and Sergey SHADRIN~$^\S$}

\AuthorNameForHeading{E.~Garcia-Failde, R.~Kramer, D.~Lewa\'{n}ski and S.~Shadrin}

\Address{$^\dag$~Institute de Physique Th\'{e}orique, CEA Paris-Saclay, Orme des Merisiers,\\
\hphantom{$^\dag$}~91191 Gif-sur-Yvette, France}
\EmailD{\href{elba.garcia-failde@ipht.fr}{elba.garcia-failde@ipht.fr}}

\Address{$^\ddag$~Max Planck Institut f\"{u}r Mathematik, Vivatsgasse 7, 53111 Bonn, Germany}
\EmailD{\href{rkramer@mpim-bonn.mpg.de}{rkramer@mpim-bonn.mpg.de}, \href{ilgrillodani@mpim-bonn.mpg.de}{ilgrillodani@mpim-bonn.mpg.de}}

\Address{$^\S$~Korteweg-de Vries Instituut voor Wiskunde, Universiteit van Amsterdam,\\
\hphantom{$^\S$}~Postbus 94248, 1090GE Amsterdam, The Netherlands}
\EmailD{\href{mailto:s.shadrin@uva.nl}{s.shadrin@uva.nl}}

\ArticleDates{Received June 19, 2019, in final form October 14, 2019; Published online October 18, 2019}

\Abstract{We employ the $1/2$-spin tautological relations to provide a particular combinatorial identity. We show that this identity is a statement equivalent to Faber's formula for proportionalities of kappa-classes on $\mathcal{M}_g$, $g\geq 2$. We then prove several cases of the combinatorial identity, providing a new proof of Faber's formula for those cases.}

\Classification{14H10; 05A10}

\renewcommand{\thefootnote}{\arabic{footnote}}
\setcounter{footnote}{0}

\section{Introduction}

The moduli spaces of curves $ \M_{g,n}$ and their Deligne--Mumford compactifications $\Mbar_{g,n}$ are central objects in modern mathematics. Although in general their Chow rings are inifinite-dimensional, there are finite-dimensional subrings, the tautological rings $ R^* $, that contain most `naturally occuring' classes. These rings have been studied since the foundational work of Mumford~\cite{Mumf83} and Faber~\cite{Faber1999}. Overviews of the main results on these rings can be found in \cite{Pand16,Tavakol2016,Vakil2003,Zvonkine2012}.

The system of tautological rings $ \{ R^*(\Mbar_{g,n})\}_{g,n}$ can be defined succinctly as the smallest system of subalgebras of the Chow rings closed under pushforwards along the three tautological maps
\begin{gather*}
\pi \colon \ \Mbar_{g,n+1} \to \Mbar_{g,n} ,\\
\rho \colon \ \Mbar_{g,n +1} \times \Mbar_{h,m+1} \to \Mbar_{g+h,n+m} ,\\
\sigma \colon \ \Mbar_{g,n+2} \times \Mbar_{g+1,n} ,
\end{gather*}
where the first map forgets the last marked point and the other two glue two marked points together, see \cite{FaPa05}. This system of rings is also closed under pullbacks along the above-mentioned maps, and it contains the natural tautological $ \psi$-, $\kappa $-, and $ \lambda $-classes, after which the rings are named.

In fact, a set of additive generators of the tautological rings can be given by dual graphs, which are graphs with $n$ leaves (or labelled half-edges) decorated as follows: to each vertex $v$ we attach a genus $ g(v)$ and a product of $ \kappa $-classes, and to each half-edge we attach a power of a $\psi$-class. Vertices represent stable components of algebraic curves, half-edges represent special points (i.e., can be either nodes of the curve or leaves), among which labelled half-edges represent the $n$ leaves.
We interpret the dual graphs as follows: we attach to vertices of genus~$g'$ and valency~$n'$ a copy of $\Mbar_{g',n'}$, we form the product of all $\kappa$- and $ \psi$-classes attached to the vertex or to its half-edges, and we push it forward along the gluing tautological maps given by the edges.

The tautological rings of the open space $ \mc{M}_{g,n}$ and its partial compactifications are defined via restriction from $ \Mgn$. As $ \mc{M}_{g,n}$ corresponds to all smooth curves, all graphs with at least one edge restrict to zero on this space; the only tautological classes are polynomials in $ \kappa$- and $\psi$-classes. We denote by $\mc{M}^{\textup{ct}}_{g,n}$ and $\mc{M}^{\textup{rt}}_{g,n}$ the partial compactifications of $\mc{M}_{g,n}$ by stable nodal curves of compact type and with rational tails, respectively.

In fact there are many relations between dual graphs in the tautological ring. These relations are called tautological relations and they encode the structure of the tautological rings. Therefore, the understanding of tautological rings boils down to the understanding of their tautological relations.

\subsection{Half-spin relations}
One way of approaching tautological relations is via cohomological field theories (CohFTs).
One particular CohFT has played a distinguished role in this context. It is a shifted version of Witten's $r$-spin class \cite{PolischukVaintrob2001, Witten1993}, and has been thoroughly studied by Pandharipande--Pixton--Zvonkine in two different ways~\cite{PPZ13,PPZ16}. On the one hand, Witten's class is quasi-homogeneous, and this gives a degree bound for its shifted version. On the other hand, any semi-simple CohFT can be constructed via Givental's action from its degree zero part, and this gives an explicit description for the shifted Witten's class, that seemingly has non-trivial contributions in high degrees. As both approaches should lead to the same result, this gives tautological relations in degrees above the bound, called $r$-spin relations.

In \cite{PPZ16}, it was also proved that the Witten $r$-spin class is polynomial in~$ r$ for~$ r$ large. This makes it possible to choose~$r$, which a priori should be an integer greater or equal to two, to be any number. In~\cite{KLLS17}, the authors observed that taking the value $ r=\frac{1}{2} $ results in much simplified relations compared to the case of general~$r$. These relations are called half-spin relations.

The coefficients of the half-spin relations are proportional to expressions of the type
\begin{gather}\label{eq:evenodd}
\binom{2a+1}{2d} \cdot (2d-1)!!, \qquad a,d \in \mathbb{Z}_{\geq 0}
\end{gather}
(cf.~\cite[Lemma 2.1]{KLLS17}). It turns out that further applications of half-spin relations require a better understanding the combinatorial structure of these numbers. We propose some purely combinatorial questions about them, cf.~Question~\ref{ques:mainquestion} and Conjecture~\ref{conj:comb} that arose naturally from our analysis of Faber's conjecture.

\subsection{Faber's intersection numbers conjecture}

The top tautological group $R^{g-2}(\mathcal{M}_g)$ is one-di\-{}men\-{}sio\-{}nal, spanned by the class $\kappa_{g-2}$, \mbox{$g\geq 2$} \cite{FaberNonVanish, Looijenga1995}. All other monomials of kappa-classes, $\kappa_{a_1}\cdots\kappa_{a_\ell}$, $\ell\geq 1$, $a_1,\dots,a_\ell \geq 1$, $a_1+\cdots+a_\ell = g-2$, are proportional to $\kappa_{g-2}$ with some coefficients of proportionality. These coefficients were conjectured by Faber in~\cite[Conjecture~1c]{Faber1999}, and he also observed in \emph{op.~cit.} that the class $\lambda_g\lambda_{g-1}$ vanishes on $\overline{\mathcal{M}}_{g,n}\setminus\mc{M}^{\textup{rt}}_{g,n}$. An equivalent form of his conjecture (now theorem) can be represented as follows:

\begin{Theorem}[Faber's intersection numbers conjecture]\label{thm:Faber}Let $n \geq 2$ and $g \geq 2$. For any $d_1,\dots, d_n\geq 1$, $d_1+\cdots+d_n = g-2+n$, there exists a constant $C_g$ that only depends on $g$ such that
\begin{gather}\label{Faber_ctt}
\frac {1}{(2g-3+n)!}
\int_{\overline{\mathcal{M}}_{g,n}} \lambda_g\lambda_{g-1} \prod_{i=1}^n \psi_i^{d_i} (2d_i-1)!! = C_g.
\end{gather}
\end{Theorem}

\begin{Remark}\label{rem:C-g} In particular, $\int_{\overline{\mathcal{M}}_{g,1}} \lambda_g\lambda_{g-1}\psi_1^{g-1}=\frac{(2g-2)!}{(2g-3)!!}C_g$. This integral is computed in~\cite[Theorem~2]{Faber1999}, so it is known that $C_g = \frac{|B_{2g}|}{2^{2g-1}(2g)!}$, where $B_{2g}$ is the Bernoulli number.
\end{Remark}

This theorem has several proofs: Getzler and Pandharipande~\cite{GetPan} derived it from the Virasoro constrains for $\mathbb{P}^2$ proved by Givental~\cite{Givental2001MMJ}. Liu and Xu~\cite{LiuXu} derived it from an identity for the $n$-point functions of the intersection numbers of $\psi$-classes that comes from the KdV equation. Goulden, Jackson, and Vakil proved it for $n\leq 3$ using the reductions of Faber--Hurwitz classes~\cite{GouJackVak}. Buryak and the fourth author proved it using relations for double ramification cycles~\cite{BuryakShadrin}. Finally, Pixton showed the compatibility of this theorem with Faber--Zagier relations in~\cite{PixtonThesis}, also proved by Faber and Zagier (unpublished, see a remark in~\cite{PPZ16}). Together with a result of~\cite{PPZ16}, this shows that Faber--Zagier relations imply this theorem.

In fact, all these independent proofs are inspired by quite different ideas and they all lead to a deeper understanding of the geometry of the moduli spaces of curves. In this paper, we use the half-spin relations to transform Faber's conjecture into a combinatorial identity. This gives insight into the use of half-spin relations and the related combinatorics of expressions of the form of (\ref{eq:evenodd}).
On the other hand, it gives insight into Faber's formula itself, as we extend it to formal negative powers of $\psi$-classes.

We then prove several cases of the combinatorial identity, providing a new proof of Faber's conjecture for $n$ less than or equal to five.

\subsection{Organization of the paper}
In Section~\ref{sec:half-spin}, we give the definition of the half-spin relations. In Section~\ref{sec:reduction}, we reduce Faber's conjecture (Theorem~\ref{thm:Faber}) to a combinatorial identity using the half-spin relations. In Section~\ref{sec:NegativePsiClasses}, we introduce formal negative powers of $\psi$-classes to reduce the combinatorial identity to a~simpler one, which we refer to as the main combinatorial identity of the paper. In Section~\ref{sec:CombStruc}, we investigate this identity from a combinatorial viewpoint and conjecture a refinement. In Appendix~\ref{sec:ProofSpecialCases}, we give a combinatorial proof of the identity in low-degree cases.

\section{Definition of half-spin relations}\label{sec:half-spin}

We will define two specific cases of the half-spin relations in $R^{\geq g} (\mc{M}^{\textup{ct}}_{g,n}) $, as this is all we need for the rest of the paper. For a more general version and the construction, see~\cite{KLLS17}.

First we need to define stable graphs.

\begin{Definition}A \emph{stable graph} is the data $\Gamma = (V, H, L,E, g\colon V \to \Z_{\geq 0}, v\colon H \to V, \iota \colon H \to H)$ such that
\begin{enumerate}\itemsep=0pt
\item[1)] $V$ is the vertex set with genus function $g$;
\item[2)] $\iota $ is an involution of $H$, the set of half-edges;
\item[3)] the set $L$ of legs or leaves is given by the fixed points of $\iota$;
\item[4)] the set $E$ of edges is given by the two-point orbits of $ \iota$;
\item[5)] $v$ sends a half-edge to the vertex it is attached to;
\item[6)] the graph given by $(V,E)$ is connected;
\item[7)] for each vertex $ w \in V$, the stability condition holds: $ 2g(w) -2 + n(w) >0$, where $n(w) = |v^{-1}(w)|$ is the valence of $w$.
\end{enumerate}
For such a stable graph, its \emph{genus} is given by $g(\Gamma ) = \sum_{v \in V} g(v) + h^1(\Gamma )$, where $h^1(\Gamma )$ is the first Betti number of the graph. The \emph{type} of a stable graph $\Gamma$ is given by $ (g(\Gamma), |L|)$.
\end{Definition}

We recall that the $r$-spin relations are proved in \cite{PPZ16} by taking the Cohomological Field Theory given by Witten's $r$-spin class, and showing that it is polynomial in $r$ in a certain way. This is a subtle argument, hinging on the primary fields $ a_1, \dotsc, a_n$ attached to the leaves. For the $r$-spin theory, these are numbers between $0$ and $r-2$, such that $ A \coloneqq \sum a_i \equiv g-1+D \mod r-1$, where $D$ is the degree of the relation. To show polynomiality in $r$, this congruence is lifted to an equality $ A = g-1+D + x(r-1)$ for some $x\in \Z_{\geq 0}$. If $x=0$, all the primary fields can be taken constant in $r$, and polynomiality follows from the argument of~\cite{PPZ16}.

For $x\geq 1$, however, the argument is more complicated, as the polynomiality does not hold over all of $ \Mgn$. However, under certain conditions, it still holds on certain subspaces, where we can then use it to get half-spin relations on these subspaces. As taking $r=\frac{1}{2}$ is in effect taking a linear combination of relations for integer graphs, we do get relations on all of $\Mgn$, but their description is not explicit outside the given subspace. For more details, see~\cite{KLLS17}. We will only give the half-spin relations needed for this paper; they only use trees.

\begin{Definition}Define the polynomials
\begin{gather}
Q_m (a) \coloneqq \frac{(-1)^m}{2^m m!} \prod_{k=1}^{2m} \left( a+1 -\frac{k}{2} \right) .
\end{gather}
Let $ n\geq 2$, $D\geq g$ and $ a_1, \dotsc, a_n$ be non-negative integers, called \emph{primary fields}, with sum $A \coloneqq \sum_{i=1}^n a_i = g-1+D$. Consider all stable \emph{trees} $ \Gamma = (V,H,L,E,g,v,\iota )$ of type $(g,n)$ and decorate them in the following way:
\begin{itemize}\itemsep=0pt
\item On each leg labeled by $i$, place the sum $ \sum_{d_i = 0}^{a_i} Q_{d_i} (a_i)\psi_i^{d_i} $, and place the integer $ a_i-d_i$ on the corresponding half-edge fixed by $\iota$.
\item On each vertex $v$, we use the tree structure to work inwards from the leaves. If we have determined all half-integers $b_i$ at its incident half-edges except one, say $ b_0$, then $b_0 \coloneqq g(v) - 1 - \sum_i b_i$ if this is at least zero. Otherwise, set $b_0 \coloneqq g(v)-\frac{3}{2} - \sum_i b_i $.
\item On each edge with half-integers $ a$ and $b$ on its two half-edges, place the sum $-\sum_{m>0} Q_{n}(a+m) (\psi + \psi')^{m-1} \delta_{a+b+m,-\frac{3}{2}}$, where $\psi$ and $\psi'$ are the $\psi$-classes corresponding to the two half-edges.
\end{itemize}
The \emph{half-spin relation for $ x= 0$}, $ \Omega_{g,n}^D(a_1, \dotsc, a_n) =0 \in R^D (\mc{M}_{g,n}^{\textup{rt}})$, is given by the sum of these decorated stable graphs with these coefficients being zero in degree $D$.
\end{Definition}
\begin{Remark}
Although the coefficient on the edge does not seem to be symmetric in $a$ and $ b$, a simple calculation shows it actually is.

In fact, the coefficient on an edge with $a$ and $ b$ on its two half-edges coming from the $r$-spin relations is
\begin{gather}\label{EdgeContributionComplicated}
\frac{1}{\psi + \psi'} \bigg( \delta_{a+b, -\frac{3}{2}} - \sum_{m,m'=0}^\infty \sum_{c,d \in \frac{1}{2}\Z} Q_m(c)Q_{m'}(d) \delta_{a,c-m} \delta_{b,d-m'} \delta_{c+d,-\frac{3}{2}} \psi^m (\psi')^{m'} \bigg) .
\end{gather}
This is equal to the coefficient given in the definition, but we give this equation as well, as it is closer to the form of the $r$-spin relations in \cite{PPZ16}, and because it is useful for the rest of the paper. In this formula, the numbers $c$ and $d$ should be interpreted as being placed near the middle of the edge, or at the end of the half-edges. In this way, they are similar to the $a_i$ on the leaves, and they will also be called primary fields. Meanwhile, the $ a_i -d_i$ are similar to the $a $ and $ b$ on the edges. This analogy will be used in the proof of Proposition~\ref{CombIdentity}. The equality can be seen from the relation
\begin{gather*}
Q_m(a+m) Q_{m'}(b+m') = \binom{m+m'}{m}Q_{m+m'}(a+m+m') \qquad \text{if $a + b + m +m' = -\frac{3}{2}$.}
\end{gather*}
\end{Remark}
\begin{Remark} These relations have been proved in \cite{KLLS17}, by specialization of the $r$-spin relations proved in~\cite{PPZ16}. The proof in~\cite{KLLS17} uses the polynomiality of the $r$-spin relations in $r$, which is also proved in~\cite{PPZ16}. The half-spin relations can be extended to all of $ \Mgn$, but this extension is not unique, much less explicit, and unnecessary for our purpose. However, their extension is of principal importance for applications in Gromov--Witten theory, since it allows to prove the following statement (a reformulation of~\cite[Lemma~5.2]{KLLS17}):
\end{Remark}
\begin{Proposition}\label{prop:lemma5.2} Any monomial of $\psi$-classes of degree at least $\max(g,1)$ on $\Mgn$ can be expressed in terms of the boundary classes that involve no $\kappa$-classes, that is, in terms of the dual graphs with at least one edge, decorated only by $\psi$-classes.
\end{Proposition}
This reformulation of~\cite[Lemma~5.2]{KLLS17} is noted in~\cite{CladerJandaWangZakharov} (where an alternative approach to the same statement is developed), and in this reformulation Proposition~\ref{prop:lemma5.2} immediately resolves Conjecture~3.14 in~\cite{LinZhou} and Conjecture~3 in~\cite{FaPa05}.

We will also need the half-spin relation on $ \Mbar_{0,n}$ for $x=1$. We give them here on $ \mc{M}_{g,n}^{\textup{ct}} $ for general $g$, which reduces to $ \Mbar_{0,n}$ for $g=0$.

\begin{Definition}Now, let $ n\geq 2$, $D \geq g+1$, and the primary fields $ a_1, \dotsc, a_{n-1}$ be non-negative integers, and $a_n \leq -\frac{3}{2}$ with sum $ A = g+ D-\frac{3}{2}$. Then the \emph{half-spin relation for $x=1$}, $ \Omega_{g,n}^D(a_1, \dotsc, a_n) =0 \in R^D (\mc{M}_{g,n}^{\textup{ct}})$, is given by a sum over decorated stable trees with the same conditions as the ones for $x=0$.
\end{Definition}
\begin{Remark}Although the (local) conditions are the same, the (global) relations are different, because the sum of the primary fields is different.
\end{Remark}

\section{A combinatorial identity from half-spin relations}\label{sec:reduction}

In this section, we employ the half-spin relations to prove the following proposition. We shall denote by $\llbracket n \rrbracket$ the set $\{1,\dots,n\}$.

\begin{Proposition}\label{CombIdentity} For any $g\geq 2$ and $n\geq 2$, for any $a_1,\dots,a_n\in\mathbb{Z}_{\geq 0}$, $a_1+\cdots+a_n=2g-3+n$, we have the following equation:
\begin{align} \label{eq:keyeq0-1}
& 0 = \sum_{k=1}^n \frac{(-1)^k}{k!}\sum_{\substack{I_1\sqcup\cdots\sqcup I_k = \llbracket n \rrbracket\\ I_j\neq \varnothing, \forall\, j\in \llbracket k \rrbracket}}
\sum_{\substack{d_1,\dots,d_k \in\mathbb{Z}_{\geq 0} \\ d_1+\cdots+d_k=g-2+k}}
\langle \tau_{d_1}\cdots \tau_{d_k}\rangle_g \cdot \prod_{j=1}^k Q_{d_j+|I_j|-1}(a_{[I_j]}).
\end{align}
Here we denote $\sum_{\ell\in I_j} a_\ell$ by $a_{[I_j]}$ and
\begin{gather}
\langle \tau_{d_1}\cdots \tau_{d_k}\rangle_g \coloneqq \frac{1}{C_g}\int_{\overline{\mathcal{M}}_{g,k}} \lambda_g\lambda_{g-1} \prod_{i=1}^k \psi_i^{d_i},
\end{gather}
where $C_g$ is an arbitrary constant depending only on $g$ $($for instance, it is convenient to assume that $C_g$ is the constant given in Remark~{\rm \ref{rem:C-g})}.

Moreover, for a fixed $g\geq 0$, the whole system of equations~\eqref{eq:keyeq0-1} $($we can vary parameters $n\geq 2$ and $a_1,\dots,a_n)$ determines all integrals $\langle \tau_{d_1}\cdots \tau_{d_k}\rangle_g$, $k\geq 2$, in terms of $\langle \tau_{g-1} \rangle_g$.
\end{Proposition}

\begin{Remark} Note that equation~\eqref{eq:keyeq0-1} can also be considered as an equation for the classes in $R^{g-2}(\mc{M}_g)$, once we replace the symbols $\langle \tau_{d_1}\cdots \tau_{d_k}\rangle_g$ by the restrictions to $R^{g-2}(\mathcal{M}_g)$ of the classes $\pi_*(\psi_1^{d_1}\cdots \psi_k^{d_k})$, where $\pi\colon \overline{\mathcal{M}}_{g,k}\to \overline{\mathcal{M}}_{g}$.
\end{Remark}

\begin{proof}We will use relations in $R^{g-2+n}(\mc{M}^{\textup{rt}}_{g,n})$ given by half-spin relations for $A=2g-3+n$. Note that to produce relations $D\coloneqq g-2+n$ must be at least $g$, and hence we have $n\geq 2$.

The restriction to $ \mc{M}^{\textup{rt}}_{g,n}$ means that all allowed stable trees must have one vertex~$v_g$ of genus~$g$, and all other vertices have genus $0$. If we cut $v_g$ from such a stable tree, it falls apart in several connected components, which are called rational tails.

The leaves are then distributed among these rational tails, and this gives a decomposition $\llbracket n \rrbracket=\bigsqcup_{i=1}^k I_i$. If $|I_i| = 1$, this corresponds to a leaf attached to $v_g$. We will therefore consider all graphs where the points with indices in $I_i$ lie on a separate rational tail, for every $i=1,\dots,k$.

We want to simplify these relations by applying half-spin relations in genus zero to each of the tails. Hence, we will now consider a particular rational tail that contains points with indices in $I\subset \llbracket n \rrbracket$, with $|I|\geq 2$. Consider the edge that attaches this rational tail to the genus $g$ component, and assume that it is decorated by $\psi^d$ at the node on the genus $g$ component. We call this edge the root edge, $e_r$, for this tail.

The total (cohomological) degree of the rest of this tail is given by the number of edges, excluding this one, together with the total number of $\psi$-classes, excluding this one. We will call this degree $D_I$. It cannot be larger than $|I|-2$, since $\dim_{\mathbb{C}}\overline{\mathcal{M}}_{0,|I|+1} = |I|-2$ and the graph is constructed via pushforward along a map from this space. This means that the end of the root edge which connects to the rational tail is decorated with $\psi^\ell$ for some $0\leq \ell\leq |I|-2$.
\begin{figure}[h!]\centering
\includegraphics{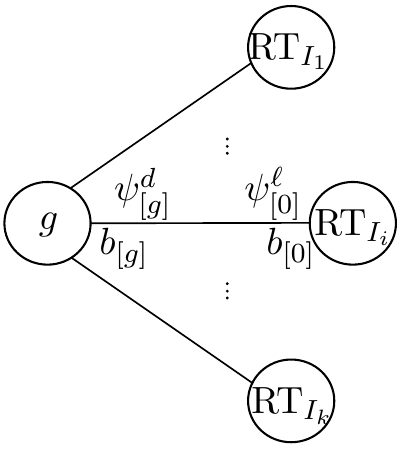}
\caption{A dual graph of genus $g$ with rational tails and $n$ leaves. The marked points with indices in $I_i\subset \llbracket n \rrbracket$ are attached to the rational tail denoted by $\text{RT}_{I_i}$, for all $i=1,\ldots,k$.}
\end{figure}

Let us now discuss the coefficient corresponding to the root edge. Using the congruences for the primary fields for the leaf contributions and the vertex contributions to be non-zero, together with the fact that all vertices in the rational tail correspond to genus $0$ components and the total number of remaining $\psi$ classes and edges is equal to $D_I-\ell$, the primary field at the genus $0$ end of the root edge must be equal to $b_{[0]}\coloneqq -\frac{3}{2}-a_{[I]}+(D_I-\ell)$. The primary field at the genus $g$ end of the root edge must be equal to $b_{[g]}\coloneqq a_{[I]}-(D_I+d+1)$.

The coefficient of the contribution of the root edge reads
\begin{gather*}
-\sum_{\ell=0}^{D_I} \psi_{[g]}^d\psi_{[0]}^\ell \bigg[
Q_{d+1}(\aoverbrace[L1U1R]{a_{[I]}-D_I}^{b_{[g]}+d+1}) Q_{\ell}\bigg(\aoverbrace[L1U1R]{{-}\frac{3}{2}-a_{[I]}+D_I}^{b_{[0]}+\ell}\bigg)\\
\qquad{} -Q_{d+2}(a_{[I]}-D_I+1) Q_{\ell-1}\left(-\frac{3}{2}-a_{[I]}+D_I-1\right)\\
\qquad{} +Q_{d+3}(a_{[I]}-D_I+2) Q_{\ell-2}\left(-\frac{3}{2}-a_{[I]}+D_I-2\right)\\
\qquad{} \vdots\\
\qquad{} + (-1)^{\ell}
Q_{d+\ell+1}(a_{[I]}-D_I+\ell) Q_{0}\left({-}\frac{3}{2}-a_{[I]}+D_I-\ell\right)\bigg],
\end{gather*}
where the alternating sum comes from the division by $ \psi_{[g]} + \psi_{[0]} $, following equation~\eqref{EdgeContributionComplicated}. Let us take this sum in a bit different way, with respect to the argument of the second factor $a_0=-\frac{3}{2}-a_{[I]}+D_I-j$, where $j$ runs from $D_I$ to $0$, and decompose the exponent of $\psi_{[0]}$ as $\ell=j+k$. We have
\begin{gather*} -\sum_{a_0=-3/2-a_{[I]}}^{-3/2-a_{[I]}+D_I} \psi_{[g]}^d (-1)^{-3/2-a_0-a_{[I]}+D_I} Q_{d+1-3/2-a_{[I]}+D_I-a_0}\left( -\frac{3}{2}-a_0\right)\psi_{[0]}\!\aoverbrace[L1U1R]{^{-3/2-a_{[I]}+D_I-a_0}}^{j}\\
\qquad{} \times \left[\sum_{k=0}^{a_{[I]}- (-3/2-a_0)} Q_k(a_0) \psi_{[0]}^k \right].
\end{gather*}
The sum over $a_0$ here is over half-integers, with integer steps.

Let us analyse the sum $\sum_{k=0}^{a_{[I]} + 3/2+a_0} Q_k(a_0) \psi_{[0]}^k$. We cut the root edge and assign $a_0$ as primary field for the new leaf on the rest of the tail, which is decorated with $\psi_{[0]}^k$. The total dimension of the class on the rest of the tail is $D_0 = D_I-j = D_I-\big({-}\frac{3}{2}-a_{[I]}+D_I-a_0\big) =\frac{3}{2}+a_{[I]}+a_0$. Thus $a_0+a_{[I]} = D_0 -\frac{3}{2}$. Therefore, if $D_0\geq 1$, then with this sum on the root edge the total sum of all graphs in the tail (for a~fixed~$a_0$) is the half-spin relation for $x=1$, with primary field~$a_0$ at the root edge and~$a_i$, $i\in I$, for the marked points on the tail.

Thus the only nontrivial contribution of the tail comes from the case $D_0=0$ which produces no relation for the tail, with $a_0=-\frac{3}{2}-a_{[I]}$. In this case there is the unique non-trivial summand in the sum above that is equal to
\begin{gather*}
(-1)^{D_I+1}\psi_{[g]}^d\psi_{[0]}^{D_I} Q_{d+D_I+1}(a_{[I]}) .
\end{gather*}
Moreover, the only non-trivial $\psi$-classes are on the root edge and there are no more internal edges on the tail.

In the end, modulo the relations in genus $0$ on the tails, the only graphs that remain in the relation in degree $D=g-2+n$ are the following. The marked points are split in $k$ non-empty sets $I_1, \dots, I_k$, corresponding to different rational tails. If $I_i$ is a set of one element, then the tail is just a leaf decorated with $\psi^{d_i}$ and the coefficient is $Q_{d_i}(a_{[I_i]})$. If $I_i$ is a set of two or more points, then this tail is just one rational vertex with all leaves from $I_i$ on it, attached by an edge to the genus $g$ vertex. The $\psi$-classes are only on this edge, $\psi^{d_i}$ on the genus $g$ side and $\psi^{D_I}$ on the genus $0$ side, with the coefficient $(-1)^{D_I+1}Q_{d_i+D_I+1} (a_{[I_i]})$.

Up to now, everything we described was done in $R^{g-2+n}(\mc{M}^{\textup{rt}}_{g,n})$. Hence, we still need to pushforward to $\mc{M}_{g,k}$, along the map forgetting some of the marked points. For each decorated graph we constructed, we will pushforward until each tail has exactly one marked point left, and hence must be a leaf.

We can do this on each tail individually, first using the string equation, which in this case reads $\int_{\Mbar_{0,|I|+1} }\psi_{[0]}^{D_I} = \int_{\Mbar_{0,|I|}} \psi_{[0]}^{D_I-1}$. Therefore, pushing forward along a map forgetting a point in $I$ decreases the exponent of $\psi_{[0]}$ by one. As this can be done until the rational tail has two marked points, we must get $ D_I = |I|-2$.

Finally, the pushforward of a rational tail with two marked points along the map forgetting one of those marked points just collapsed the tail and moves the remaining marked point to the collapsed node.

Summarising, the only surviving terms are the terms where all the marked points are partitioned as $ \bigsqcup_k I_k = \llbracket n \rrbracket $ over rational tails consisting of a leaf or a single rational curve with all marked points attached to it, with coefficient
\begin{gather*}
(-1)^{|I|-1}\psi_{[g]}^d\psi_{[0]}^{|I|-2} Q_{d+|I|-1}(a_{[I]}) .
\end{gather*}
These terms pushforward to terms on $ \mc{M}_{g,k}$ given by
\begin{gather*}
(-1)^{|I|-1}\psi_I^d Q_{d+|I|-1}(a_{[I]}) .
\end{gather*}
Taking the product over all the tails and taking into account that the linear function
\begin{gather*}\frac{1}{C_g}\int_{\overline{\mathcal{M}}_g}\lambda_g\lambda_{g-1}\cdot \colon \ R^{g-2}(\mathcal{M}_g)\rightarrow \mathbb{Q}
\end{gather*}
is an isomorphism, the half-spin relations we found for $D=g-2+n$ imply the combinatorial identity~\eqref{eq:keyeq0-1}.

On the other hand, it is easy to see that these relations determine the intersections of all possible monomials in $\psi$-classes in terms of $\int_{\overline{\mathcal{M}}_{g,1}}\lambda_{g}\lambda_{g-1}\psi_1^{g-1}$ (using the natural lexicographic order).
\end{proof}

We relate Proposition~\ref{CombIdentity} to Faber's conjecture and refine it using the string equation, which turns the result into a combinatorial identity.

\begin{Corollary}Let $g \geq 2$ and $n \geq 2$. The following two statements are equivalent:
\begin{enumerate}\itemsep=0pt
\item[$i)$] Faber's Conjecture~{\rm \ref{thm:Faber}}: there exists a constant $C_g$ that only depends on $g$ such that
\begin{gather}\label{eq:FabersConj}
\langle \tau_{d_1}\cdots \tau_{d_k}\rangle_g \coloneqq \frac{1}{C_g} \int_{\overline{\mathcal{M}}_{g,n}} \lambda_g \lambda_{g-1} \prod_{i=1}^n \psi_i^{d_i} = \frac{(2g - 3 + n)!}{\prod\limits_{i=1}^n (2d_i - 1)!!}
\end{gather}
for any $d_1, \dots, d_n \geq 1$.
\item[$ii)$] For any $a_1,\dots,a_n\in\mathbb{Z}_{\geq 0}$ such that $a_1+\cdots+a_n=2g-3+n$, we have
\begin{align} \label{eq:keyeq0}
& 0 = \sum_{k=1}^n \frac{(-1)^k}{k!}
\sum_{\substack{I_1\sqcup\cdots\sqcup I_k = \llbracket n \rrbracket\\ I_j\neq \varnothing, \forall\, j\in \llbracket k \rrbracket}}
\sum_{\substack{d_1,\dots,d_k \in\mathbb{Z}_{\geq 0} \\ d_1+\cdots+d_k=g-2+k}}
\langle \tau_{d_1}\cdots \tau_{d_k}\rangle_g^{\star}\cdot\prod_{j=1}^k Q_{d_j+|I_j|-1}(a_{[I_j]}),
\end{align}
where
\begin{gather}
\langle \tau_{d_1}\cdots \tau_{d_k}\rangle_g^{\star}=
\frac{(2g - 3 + k)!}{\prod\limits_{i=1}^k (2d_i - 1)!!} \qquad \text{in case} \quad d_1, \dots, d_k \geq 1,
\end{gather}
and determined by the string equation
\begin{gather}\label{eq:string}
\langle \tau_{d_1}\cdots \tau_{d_k}\tau_0\rangle_g^{\star} = \sum_{\substack{j=1 \\ d_j\geq 1}}^k \langle \tau_{d_1-\delta_{1j}}\cdots \tau_{d_k-\delta_{kj}}\rangle^{\star}_g
\end{gather}
otherwise. Here $a_{[I_j]}$ denote $\sum_{\ell\in I_j} a_\ell$.
\end{enumerate}
\end{Corollary}

\begin{proof}By Proposition~\ref{CombIdentity}, the left-hand side of equation~\eqref{eq:FabersConj} satisfies equation~\eqref{eq:keyeq0}, and the integrals can be recovered from this equation. Therefore, both sides of equation~\eqref{eq:FabersConj} are equal if and only if the right-hand side also satisfies equation~\eqref{eq:keyeq0}.
\end{proof}

\section{Psi-classes of negative degree}\label{sec:NegativePsiClasses}

In the previous section, we showed that Theorem~\ref{thm:Faber} is equivalent to a system of combinatorial identities. The goal of this section is to reduce this system to a much nicer system of identities. In order to do this, we need to consider formal systems of correlators satisfying the string equation.

\subsection{Formal negative degrees of psi-classes}

\begin{Definition}\label{def:stringequation}Let $g\geq 2$. Consider a system of numbers $\big\langle\prod_{i=1}^k \tau_{d_i} \big\rangle_g^\bullet$ that depends on $d_1,\dots,d_k\in \mathbb{Z}$, $d_1+\dots+d_k = g-2+k$, and is symmetric in these variables. We say that this system of numbers \emph{satisfies the string equation} if
\begin{gather*}
\langle \tau_{d_1}\cdots \tau_{d_k}\tau_0\rangle_g^\bullet = \sum_{j=1}^k \langle \tau_{d_1-\delta_{1j}}\cdots \tau_{d_k-\delta_{kj}}\rangle_g^\bullet.
\end{gather*}
\end{Definition}

\begin{Example} \label{example:1}The system of numbers
\begin{gather*}
\bigg\langle \prod_{i=1}^k \tau_{d_i} \bigg\rangle_g^\bullet \coloneqq \begin{cases}
\displaystyle \bigg\langle\prod_{i=1}^k \tau_{d_i} \bigg\rangle_g^{\star}, & d_1,\dots, d_k\geq 0, \\
0, & \text{at least one } d_i \text{ is negative}
\end{cases}
\end{gather*}
satisfies the string equation, as follows, by definition, from equation~\eqref{eq:string}.
\end{Example}

\begin{Example} \label{example:2}The system of numbers $\big\langle \prod_{i=1}^k \tau_{d_i} \big\rangle_g^\bullet \coloneqq (2g-3+k)!/\prod_{i=1}^k (2d_i-1)!!$
also satisfies the string equation (this can be checked by direct inspection).
\end{Example}

\begin{Remark}These two examples coincide in the case when all $d_i$'s are positive and also in the case when all $d_i$'s except for one are positive and the remaining one is equal to zero. For other values of $(d_1,\dots,d_k)$ the numbers in these two examples are generally different.
\end{Remark}

The string equation allows to choose the values of all numbers $\big\langle\prod_{i=1}^k \tau_{d_i} \big\rangle_g^\bullet$, $\prod_{i=1}^{k} d_i \not=0$, $k\geq 1$ in an arbitrary way, and the rest of the numbers (where at least one index $d_i$ is equal to zero) are linear combinations of these initial values with non-negative integer coefficients.

\subsection[$Q$-polynomials and a refined string equation]{$\boldsymbol{Q}$-polynomials and a refined string equation}

Fix $g\geq 2$ and $n\geq 2$ and let $a_1,\dots,a_n$ be formal variables. Define $Q_i(a)\equiv 0 $ for $i<0$. Fix an arbitrary system of numbers $\big\langle\prod_{i=1}^k \tau_{d_i} \big\rangle_g^\bullet$, $d_1,\dots,d_k\in \mathbb{Z}$, $d_1+\dots+d_k = g-2+k$, symmetric in these variables and satisfying the string equation.

Consider the following expression
\begin{gather} \label{eq:expression}
\mathcal{E}_{g,n}(\vec{a}) \coloneqq
 \sum_{k=1}^n \frac{(-1)^k}{k!}
\sum_{\substack{I_1\sqcup\cdots\sqcup I_k = \llbracket n \rrbracket\\ I_j\neq \varnothing,\, \forall\, j\in \llbracket k \rrbracket}}
\sum_{\substack{d_1,\dots,d_k \in\mathbb{Z} \\ d_1+\cdots+d_k=g-2+k}}
 \langle \tau_{d_1}\cdots \tau_{d_k}\rangle_g^\bullet \cdot \prod_{j=1}^k Q_{d_j+|I_j|-1}(a_{[I_j]})
\end{gather}
as a polynomial in $a_1,\dots,a_n$ and a linear function in $\langle \tau_{d_1}\cdots \tau_{d_k}\rangle_g^\bullet$, $d_1\cdots d_k\not=0$.

\begin{Proposition}\label{prop:derivative} For any $d_1,\dots,d_k\in \mathbb{Z}$, $d_1+\dots+d_k = g-2+k$, $d_1\cdots d_k\not=0$, where at least one index $d_i$ is negative, we have
\begin{gather*}
\frac{\partial \mathcal{E}_{g,n}(\vec{a})}{\partial \langle \tau_{d_1}\cdots \tau_{d_k}\rangle_g^\bullet} \equiv 0.
\end{gather*}
\end{Proposition}

\begin{proof} Assume $d_1,\dots,d_\ell$ are negative, and the rest of the indices $d_i$ are positive. Let us fix $I_1\sqcup\cdots\sqcup I_k\subset \llbracket n \rrbracket$ that satisfy the condition $|I_i|+d_i-1\geq 0$ for any $i=1,\dots,\ell$.

Consider all terms in the expression $\mathcal{E}$ satisfying the following conditions:
\begin{itemize}\itemsep=0pt
\item The correlator factor is a coefficient $\big\langle \prod_{i=1}^k \prod_{j=1}^{m_i} \tau_{d_{ij}} \big\rangle_g^\bullet$ such that
\begin{itemize}\itemsep=0pt
\item for each $i=1,\dots,k$, we have $\sum_{j=1}^{m_i} d_{ij} = d_i + m_i- 1$;
\item for each $i=1,\dots,\ell$, at most one of $d_{ij}$, $j=1,\dots, m_i$ is negative. It is at least $ d_i$, and at least $ d_i+1$ if $m_i \geq 2$ (this ensures there exists a zero index to use the string equation);
\item for each $i=\ell+1,\dots,k$, all $d_{ij}$, $j=1,\dots, m_i$ are non-negative. Moreover, one index must be at least $d_i$, and at least $d_i+1$ if $m_i\geq 2$.
\end{itemize}
This list of conditions is equivalent to
$\frac{\partial \big\langle \prod_{i=1}^k \prod_{j=1}^{m_i} \tau_{d_{ij}} \big\rangle_g^\bullet}{\partial \langle \tau_{d_1}\cdots \tau_{d_k}\rangle_g^\bullet}
\not= 0$. In other words, the correlator in the denominator can be deduced from the one in the numerator via successive applications of the string equation.
\item The sets $I_{ij}$ satisfy $\sqcup_{j=1}^{m_i} I_{ij} = I_i$ for each $i=1,\dots,\ell$.
\item For each $i=1,\dots,k$ the sets $I_{ij}$ are arranged in such a way that $\min(I_{ij})<\min(I_{ij'})$ if and only if $j<j'$ (this condition is necessary to have control on the combinatorial factor).
\end{itemize}
We can refine (\ref{eq:expression}) as follows: we define ``refined correlators'' $\big\langle\prod_{i=1}^k \tau_{d_i} (J_i)\big\rangle_g^\bullet$, now depending formally on subsets $J_i \subset \llbracket n \rrbracket$, and subject to a natural refinement of the string equation
\begin{gather*}
\langle\tau_0(J_{k+1})\prod_{i=1}^k \tau_{d_i} (J_i)\rangle_g^\bullet =
\sum_{j=1}^k \langle\tau_{d_j-1}(J_j\sqcup J_{k+1})\prod_{\substack{i=1\\ i\not=j}}^k \tau_{d_i} (J_i)\rangle_g^\bullet .
\end{gather*}
We then define $\mathcal{E}^\textup{ref}_{g,n}(\vec{a})$ to be
\begin{gather*}
\mathcal{E}^\textup{ref}_{g,n}(\vec{a}) \coloneqq
 \sum_{k=1}^n \frac{(-1)^k}{k!}
\sum_{\substack{I_1\sqcup\cdots\sqcup I_k = \llbracket n \rrbracket\\ I_j\neq \varnothing, \, \forall\, j\in \llbracket k \rrbracket}}
\sum_{\substack{d_1,\dots,d_k \in\mathbb{Z} \\ d_1+\cdots+d_k=g-2+k}}
\langle \tau_{d_1}(I_1)\cdots \tau_{d_k}(I_k)\rangle_g^\bullet\cdot\prod_{j=1}^k Q_{d_j+|I_j|-1}(a_{[I_j]}).
\end{gather*}
Clearly, $\mathcal{E}^\textup{ref}_{g,n}(\vec{a})$ reduces to $\mathcal{E}_{g,n}(\vec{a})$ after setting $\tau_d (I) \to \tau_d$.

Using this notation, if we fix $m_i$, $d_{ij}$ and $I_{ij}$ for $i>1$, and let $m_1$, $d_{1j}$, and $I_{1j}$ vary in all possible ways such that the conditions above are satisfied, we can split the derivative $
{\partial \langle \prod_{i=1}^k \prod_{j=1}^{m_i} \tau_{d_{ij}} \rangle_g^\bullet}/
{\partial \langle \tau_{d_1}\cdots \tau_{d_k}\rangle_g^\bullet}$ into the sum of ``refined derivatives''
\begin{gather*}
 \frac
{\partial \Big\langle \prod\limits_{i=1}^k \prod\limits_{j=1}^{m_i} \tau_{d_{ij}}(I_{ij}) \Big\rangle_g^\bullet}
{\partial \Big\langle \prod\limits_{i=1}^k \tau_{d_i}(I_i)\Big\rangle_g^\bullet}
= \frac
{\partial \Big\langle \prod\limits_{i=1}^k \prod\limits_{j=1}^{m_i} \tau_{d_{ij}}(I_{ij}) \Big\rangle_g^\bullet}
{\partial \Big\langle \tau_{d_1}(I_1) \prod\limits_{i=2}^k \prod\limits_{j=1}^{m_i} \tau_{d_{ij}}(I_{ij}) \Big\rangle_g^\bullet}
\frac
{\partial \Big\langle\tau_{d_1}(I_1) \prod\limits_{i=2}^k \prod\limits_{j=1}^{m_i} \tau_{d_{ij}}(I_{ij}) \Big\rangle_g^\bullet}
{\partial \Big\langle \prod\limits_{i=1}^k \tau_{d_i}(I_i)\Big\rangle_g^\bullet} .
\end{gather*}
The derivative is clearly zero if the partition $\{I_{ij}\}$ is not a refinement of the partition $\{I_i\}$.

Thus we obtain the following expression for the derivative of $\mathcal{E}_{g,n}(\vec{a})$:
\begin{gather}
\frac{\partial \mathcal{E}_{g,n}(\vec{a})}{\partial \langle \tau_{d_1}\cdots \tau_{d_k}\rangle_g^\bullet} = \sum_{I_1\sqcup\cdots\sqcup I_k= \llbracket n \rrbracket} \frac{\partial \mathcal{E}^\textup{ref}_{g,n}(\vec{a})}{\partial \langle \tau_{d_1} (I_1) \dotsb \tau_{d_k}(I_k) \rangle^\bullet_g} \Bigg|_{\tau_d(I) = \tau_d}
\nonumber\\
= \frac{(-1)^k}{k!} \sum_{\substack{I_1\sqcup\cdots\sqcup I_k= \llbracket n \rrbracket\\ I_j\neq \varnothing, \forall\, j\in \llbracket k \rrbracket \\ m_i, d_{ij}, I_{ij} \text{ for } i\geq 2}} (-1)^{\sum\limits_{i=2}^k (m_i-1)}
\prod_{i=2}^k \prod_{j=1}^{m_i} Q_{d_{ij}+|I_{ij}|-1}(a_{[I_{ij}]}) \frac
{\partial \Big\langle \tau_{d_1}(I_1) \prod\limits_{i=2}^k \prod\limits_{j=1}^{m_i} \tau_{d_{ij}}(I_{ij}) \Big\rangle_g^\bullet}
{\partial \Big\langle \prod\limits_{i=1}^k \tau_{d_i}(I_i)\Big\rangle_g^\bullet} \nonumber\\
 \times \left(\sum_{m_1, d_{1j}, I_{1j}} (-1)^{ m_1-1} \prod_{j=1}^{m_1} Q_{d_{1j}+|I_{1j}|-1}(a_{[I_{1j}]}) \frac
{\partial \Big\langle \prod\limits_{i=1}^k \prod\limits_{j=1}^{m_i} \tau_{d_{ij}}(I_{ij}) \Big\rangle_g^\bullet}
{\partial \Big\langle \tau_{d_1}(I_1) \prod\limits_{i=2}^k \prod\limits_{j=1}^{m_i} \tau_{d_{ij}}(I_{ij}) \Big\rangle_g^\bullet}\right)\Bigg|_{\tau_d(I) = \tau_d} .\label{eq:exprderE}
\end{gather}
In order to prove the proposition, it is sufficient to show that the factor in the third line of this expression is always equal to zero. Note that this factor is a polynomial in the variables $a_p$, $p\in I_1$, of degree $2\left(d_1+m_1-1 + |I_1| - m_1\right) $. The degree of this polynomial is less than twice the number of its variables (since $d_1<0$). Therefore, in order to show the constant vanishing of this polynomial, it is sufficient to show that it constantly vanishes for two specific values of each of its variables, namely, at the points $a_p=0$ and $a_p=-1/2$ for each $p\in I_1$. Since this polynomial is symmetric in its variables, it is sufficient to prove this vanishing for just one variable.

We assume, for simplicity, that $1\in I_1$, and prove the vanishing for $a_1=0,-1/2$. In order to use the string equation, we split the terms in the third line of~\eqref{eq:exprderE} in two parts: those where $I_{1,1} = \{ 1\} $, and those where $ I_{1,1} \supsetneq \{ 1\} $ (as $ 1 \in I_{1,1} $ by the third bullet of conditions). The first part is parametrised by partitions $ I_{1,2} \sqcup \dotsb \sqcup I_{1,m_1} = I_1 \setminus \{ 1\} $, and the second part can be reparametrised by the same partitions, plus a choice of one of these sets which should also contain $1$. Hence, up to a common sign factor, we can split the third line of~\eqref{eq:exprderE} as terms of the form
\begin{gather} \label{eq:vanishing}
Q_{d_{11}+1-1}(a_1)\prod_{j=2}^{m_1} Q_{d_{1j}+|I_{1j}|-1}(a_{[I_{1j}]}) \frac
{\partial \Big\langle \prod\limits_{i=1}^k \prod\limits_{j=1}^{m_i} \tau_{d_{ij}}(I_{ij}) \Big\rangle_g^\bullet}
{\partial \Big\langle \tau_{d_1}(I_1) \prod\limits_{i=2}^k \prod\limits_{j=1}^{m_i} \tau_{d_{ij}}(I_{ij}) \Big\rangle_g^\bullet}
\\
-\sum_{r=2}^{m_1} \prod\limits_{j=2}^{m_1} Q_{d_{1j}+|I_{1j}|-1}(a_{[I_{1j}]} + a_1 \delta_{j,r})
\frac
{\partial \Big\langle \tau_{d_{1r}-1}(I_{1r}\sqcup\{1\})\prod\limits_{\substack{j=2\\ j\not=r}}^{m_1} \tau_{d_{1j}}(I_{1j})\prod\limits_{i=2}^k \prod\limits_{j=1}^{m_i} \tau_{d_{ij}}(I_{ij}) \Big\rangle_g^\bullet}
{\partial \Big\langle \tau_{d_1}(I_1) \prod\limits_{i=2}^k \prod\limits_{j=1}^{m_i} \tau_{d_{ij}}(I_{ij}) \Big\rangle_g^\bullet} .\nonumber
\end{gather}
In both the cases $a_1 = 0 $ and $a_1 = -1/2$, $d_{11}$ must be equal to zero here (otherwise $Q_{d_{11}}(a_1)=0$ in the first term and the indices $d_{ij}$ in the other terms do not add up to $d_i +m_i-1$). Then, for $a_1=0$ all $Q$-coefficients in~\eqref{eq:vanishing} are literally the same, so it vanishes using the following derivative of the refined string equation
\begin{gather}
\frac
{\partial \Big\langle \tau_0(\{ 1\} ) \prod\limits_{\substack{j=2}}^{m_1} \tau_{d_{1j}}(I_{1j})
\prod\limits_{i=2}^k \prod\limits_{j=1}^{m_i} \tau_{d_{ij}}(I_{ij}) \Big\rangle_g^\bullet}
{\partial \Big\langle \tau_{d_1}(I_1) \prod\limits_{i=2}^k \prod\limits_{j=1}^{m_i} \tau_{d_{ij}}(I_{ij}) \Big\rangle_g^\bullet}\nonumber\\
 \qquad {}=
\sum_{r=2}^{m_1}\frac
{\partial \Big\langle \tau_{d_{1r}-1}(I_{1r}\sqcup\{1\})\prod\limits_{\substack{j=2\\ j\not=r}}^{m_1} \tau_{d_{1j}}(I_{1j})\prod\limits_{i=2}^k \prod\limits_{j=1}^{m_i} \tau_{d_{ij}}(I_{ij}) \Big\rangle_g^\bullet}
{\partial \Big\langle \tau_{d_1}(I_1) \prod\limits_{i=2}^k \prod\limits_{j=1}^{m_i} \tau_{d_{ij}}(I_{ij}) \Big\rangle_g^\bullet} .\label{eq:derrefstring}
\end{gather}

The case $a_1=-1/2$ is more subtle. We use induction on $|I_1|$. Using the identity $Q_p(a)-Q_p(a-1/2) = - (a/2) \cdot Q_{p-1}(a-1)$ and the derivative of the refined string equation~\eqref{eq:derrefstring}, we can rewrite expression~\eqref{eq:vanishing} as
\begin{gather}\label{eq:12reduction}
-\frac{1}{2} \sum_{k=2}^{m_1} a_{[I_{1k}]} \cdot \prod\limits_{\substack{j=2}}^{m_1} Q_{\tilde d_{1j}+|I_{1j}|-1}(\tilde a_{[I_{1j}]})
\frac
{\partial \Big\langle \prod\limits_{\substack{j=2}}^{m_1} \tau_{\tilde d_{1j}}(I_{1j})\prod\limits_{i=2}^k \prod\limits_{j=1}^{m_i} \tau_{d_{ij}}(I_{ij}) \Big\rangle_g^\bullet}
{\partial \Big\langle \tau_{d_1}(I_1\setminus\{1\}) \prod\limits_{i=2}^k \prod\limits_{j=1}^{m_i} \tau_{d_{ij}}(I_{ij}) \Big\rangle_g^\bullet} ,
\end{gather}
where in the coefficient of $a_k$ we use the notation $\tilde a_{[I_{1p}]} \coloneqq a_{[I_{1p}]}-\delta_{pk}$, $p = 2,\dotsc,m_1 $, and $\tilde d_{1q} \coloneqq d_{1q} - \delta_{kq}$, $q=2,\dots,m_1$.

By our assumptions, $ d_1<0$ and $ |I_1| + d_1 -1\geq 0$, so $|I_1|\geq 2$. For the base case of the induction, $|I_1| = 2$, we then have $d_1 = -1$, so $ d_{11} + d_{12} = d_1+m_1-1 = 0$, and therefore $ d_{12} = 0$. In this case, the $k$-sum and $j$-product in equation~\eqref{eq:12reduction} collapse, yielding a coefficient $ Q_{\tilde d_{12}+|I_{12}|-1}(\tilde a_{I_{[12]}}) = Q_{-1}(a_{I_{[12]}}-1) = 0$, proving the basis step.

For the induction step, we see that for each $k$ the coefficient of $a_{[I_{1k}]}$ in~\eqref{eq:12reduction} is the polynomial in one fewer variables (namely, taking out the variable $a_1$) and size of $|I_1|$ one less (by remo\-ving~$1$) of exactly the same form as the summands in the third line of~\eqref{eq:exprderE}. Resumming over~$m_1$,~$d_{1j}$, and~$I_{1,j}$ with the sign coming from the third line of~\eqref{eq:exprderE}, this becomes equal to the third line of~\eqref{eq:exprderE}, which was zero by induction (where we use already that this third line vanishes for all $ a_i =0, -1/2$).

Thus, the third line in~\eqref{eq:exprderE} is equal to zero for $a_1=-1/2$ as well as $a_1 = 0$, which implies it vanishes constantly. This implies the proposition.
\end{proof}

\subsection[Applying formal negative degrees of $\psi$-classes to the combinatorial identity]{Applying formal negative degrees of $\boldsymbol{\psi}$-classes\\ to the combinatorial identity}

We use the result of the previous section to reduce the system of identities~\eqref{eq:keyeq0} to a simpler one.

\begin{Proposition} \label{prop:combequiv} Faber's conjecture $($Theorem {\rm \ref{thm:Faber})} is equivalent to the following system of combinatorial identities.

For any $g\geq 2$ and $n\geq 2$, for any $a_1,\dots,a_n\in\mathbb{Z}_{\geq 0}$, $a_1+\cdots+a_n=2g-3+n$,
\begin{gather} \label{eq:keyeq1}
0 =\sum_{k=1}^n \frac{(-1)^k}{k!}
\sum_{\substack{I_1\sqcup\cdots\sqcup I_k = \llbracket n \rrbracket\\ I_j\neq \varnothing, \forall\, j\in \llbracket k \rrbracket}}
\sum_{\substack{d_1,\dots,d_k \in\mathbb{Z} \\ d_1+\cdots+d_k=g-2+k}}
\frac{(2g-3+k)!}{\prod\limits_{i=1}^k (2d_i-1)!!}\cdot\prod_{j=1}^k Q_{d_j+|I_j|-1}(a_{[I_j]}).
\end{gather}
\end{Proposition}

\begin{proof} We have already shown that Faber's conjecture is equivalent to the system of identities~\eqref{eq:keyeq0}, where the correlators are replaced by the predicted value from the conjecture. So, it is sufficient to show that the system of identities~\eqref{eq:keyeq0} is equivalent to the system of identities~\eqref{eq:keyeq1}. Note that the right hand side of~\eqref{eq:keyeq0} is a specialization of the expression $\mathcal{E}_{g,n}(\vec{a})$ for the values of $\langle \tau_{d_1}\cdots \tau_{d_k}\rangle_g^\bullet$ given in Example~\ref{example:1}. The right hand side of~\eqref{eq:keyeq1} is a specialization of the expression $\mathcal{E}_{g,n}(\vec{a})$ for the values of $\langle \tau_{d_1}\cdots \tau_{d_k}\rangle_g^\bullet$ given in Example~\ref{example:2}.

As observed before, a system of numbers $\langle \tau_{d_1}\cdots \tau_{d_k}\rangle_g^\bullet$ satisfying the string equation, see Definition~\ref{def:stringequation}, is fixed by choosing all numbers for $ \prod_{i=1}^k d_i \neq 0$ arbitrarily and inferring the other cases from the string equation. We call the chosen numbers initial values.

The initial values of the system of numbers $\langle \tau_{d_1}\cdots \tau_{d_k}\rangle_g^\bullet$ given in Examples~\ref{example:1} and~\ref{example:2} coincide for all $d_1,\dots,d_k>0$ and differ when at least one of $d_i$'s is negative. Proposition~\ref{prop:derivative} implies that the initial values $\langle \tau_{d_1}\cdots \tau_{d_k}\rangle_g^\bullet$ with at least one $d_i$ negative have no impact on the value of $\mathcal{E}_{g,n}(\vec{a})$. Therefore, a specialization of the expression $\mathcal{E}_{g,n}(\vec{a})$ for the values of $\langle \tau_{d_1}\cdots \tau_{d_k}\rangle_g^\bullet$ given in Example~\ref{example:1} is equal to zero if and only if the same specialization of the expression $\mathcal{E}_{g,n}(\vec{a})$ for the values of $\langle \tau_{d_1}\cdots \tau_{d_k}\rangle_g^\bullet$ given in Example~\ref{example:2} is equal to zero. Therefore, the system of identities~\eqref{eq:keyeq0} is equivalent to the system of identities~\eqref{eq:keyeq1}.
\end{proof}

\section{The main combinatorial identity and its structure}\label{sec:CombStruc}

In the previous section we used formal negative degree psi-classes in order to simplify the system of combinatorial identities to which Faber's conjecture is equivalent (Proposition~\ref{prop:combequiv}). We now want to substitute the $Q$-polynomials by their definition and rearrange the terms to obtain the following statement.

\begin{Corollary}[of Proposition~\ref{prop:combequiv}]\label{cor:identity}
Faber's conjecture $($Theorem {\rm \ref{thm:Faber})} is equivalent to the following system of combinatorial identities.

For any $g\geq 2$ and $n\geq 2$, for any $a_1,\dots,a_n\in\mathbb{Z}_{\geq 0}$, $a_1+\cdots+a_n=2g-3+n$, we have
\begin{gather}
 0 =
\sum_{k=1}^n \frac{(-1)^k(2g-3+k)!}{k!}\nonumber\\
\hphantom{0 =}{}\times
\sum_{\substack{I_1\sqcup\cdots\sqcup I_k = \llbracket n \rrbracket\\ I_j\neq \varnothing, \forall\, j\in \llbracket k \rrbracket}}
\sum_{\substack{d_1,\dots,d_k \in\mathbb{Z}_{\geq 0} \\ d_1+\cdots+d_k=g-2+n}}
\prod_{j=1}^k \binom{2a_{[I_j]}+1}{2d_j} \frac{(2d_j-1)!!}{(2d_j+1-2|I_j|)!!}.\label{eq:keyeq}
\end{gather}

Here by $a_{[I_j]}$ we denote $\sum_{\ell\in I_j} a_\ell$ and by $|I_j|$ we denote the cardinality of the set $I_j\subset \{1,\dots,n\}$, $j=1,\dots,k$.
\end{Corollary}
\begin{proof}Recall that, for $m\geq 0$, $Q_m(a) = ((-1)^m / 2^{3m}m!) \cdot \prod_{i=0}^{2m-1} (2a+1-i)$. If its argument is a non-negative integer, we can rewrite $Q_m(a)$ as $\binom{2a+1}{2m}\cdot (2m-1)!! \cdot (-1)^m/2^{2m}$. For $m<0$, $Q_m(a)\equiv 0$. Then it is easy to see that equation~\eqref{eq:keyeq} is obtained from equation~\eqref{eq:keyeq1} by the relabelling $d_j + |I_j|-1 \rightsquigarrow d_j$ and dividing by a common factor of $(-1)^{g-2+n}/2^{2(g-2+n)}$.
\end{proof}

In the rest of the paper we refer to equation (\ref{eq:keyeq}) as the main combinatorial identity. This section is devoted to a purely combinatorial analysis of this identity. Clearly, since Faber's conjecture is proved, the main combinatorial identity holds true. However, we are interested in an independent proof of it, in order to obtain a new proof of Faber's conjecture. We produce such a proof for $n \leq 5$ in the next section.

\begin{Question} \label{ques:mainquestion}
Is there a purely combinatorial way to prove the combinatorial identity \eqref{eq:keyeq} for all $n$?
\end{Question}

In fact, one of the purposes of the present article is to pose the question to the combinatorial community about a possible enumerative interpretation of our identity.

\subsection{Polynomials vanishing in the integer points of some simplices}
We denote the right hand side of~\eqref{eq:keyeq} by $P(a_1,\dots,a_n)$. It can be considered as a polynomial of degree $2g-4+2n$ in $a_1,\dots,a_n$ (since the binomial coefficient $\binom{2a_{[I]}+1}{2d}$ is naturally a polynomial of degree $2d$ in $a_{[I]}$). We can also rewrite it as
\begin{gather*}
 \sum_{k=1}^n \frac{(-1)^k(2g-3+k)!}{k!}
\!\!\!\!\!\!\sum_{\substack{I_1\sqcup\cdots\sqcup I_k = \llbracket n \rrbracket\\ I_j\neq \varnothing, \,\forall\, j\in \llbracket k \rrbracket}}
\!\sum_{\substack{f_1,\dots,f_k \in\mathbb{Z}_{\geq 0} \\ f_1+\cdots+f_k=g-1}}
\!\prod_{j=1}^k \binom{2a_{[I_j]}+1}{2f_j+1} \frac{(2a_{[I_j]}-2f_j-1)!!}{(2a_{[I_j]}-2f_j+1-2|I_j|)!!},
\end{gather*}
and refer to it as $R(a_1,\dots,a_n)$. The function $R(a_1,\dots,a_n)$ is also a polynomial in $a_1,\dots,a_n$, where each term in the sum over $k=1,\dots,n$ has degree $\sum_{j=1}^k 2f_j+|I_j| = 2g-2+n$, so the total degree of $R$ is $2g-2+n$. Note that $P\not=R$ (they even have different degrees); from the construction they coincide only on the simplex $a_1,\dots,a_n\in\mathbb{Z}_{\geq 0}$, $a_1+\cdots+a_n=2g-3+n$.

\begin{Proposition}
We have $P|_{a_i=0}\equiv 0$, $i=1,\dots,n$.
\end{Proposition}

\begin{proof}
Since $P$ is symmetric in its variables, it is enough to prove this proposition for $a_n=0$. Consider an arbitrary splitting $I_1\sqcup\cdots \sqcup I_{k} =\{1,\dots,n-1\}$. We want to append this splitting with the element~$n$: either we add $\{n\}$ as one of the sets (there are $k+1$ ways to do this, since we can choose the number of this set from $1$ to $k+1$ shifting the indices of $I_j$ accordingly), or we append $n$ to one of the existing sets $I_\ell$, $\ell=1,\dots k$. Consider the sum of all terms in $P$ that correspond to these choices of splitting of $\{1,\dots,n\}$. Since the first $k+1$ terms are all equal to each other, we can assume that we have $k+1$ copies of the case $I_{k+1} = \{n\}$ instead. Therefore, if we split the terms of $P$ in this way, we get summands of the following form:
\begin{gather*}
 (k+1) \frac{(-1)^{k+1}(2g-3+k+1)!}{(k+1)!}\!\!\!\! \sum_{\substack{d_1+\dots+d_{k+1}\\=g-2+n}} \prod_{j=1}^k \binom{2a_{[I_j]}+1}{2d_j} \frac{(2d_j-1)!!}{(2d_j+1-2|I_j|)!!} \cdot \binom{2a_n+1}{2d_{k+1}}\\
\qquad{} +
 \frac{(-1)^k(2g-3+k)!}{k!}\!\!\!\!\!\! \sum_{\substack{d_1+\dots+d_{k}\\=g-2+n}} \sum_{\ell=1}^k \prod_{j=1}^k \binom{2a_{[I_j]}+2a_n\delta_{j,\ell}+1}{2d_j} \frac{(2d_j-1)!!}{(2d_j+1-2|I_j|)!!}.
\end{gather*}
If $a_n=0$, then the first summand is nontrivial only for $d_{k+1}=0$. So, if we substitute $a_n=0$, then this expression is equal to
\begin{gather*}
 \frac{(-1)^k(2g-3+k)!}{k!}\sum_{\substack{d_1+\dots+d_{k}\\=g-2+n}}
\prod_{j=1}^k \binom{2a_{[I_j]}+1}{2d_j} \frac{(2d_j-1)!!}{(2d_j+1-2|I_j|)!!} \\
\qquad{}\times \left(-(2g-2+k)+ \sum_{\ell=1}^k (2d_\ell+1-2|I_\ell|)\right).
\end{gather*}
Note that the last factor is equal to zero. Since the definition of $P$ reduces to the sum over $I_1\sqcup\cdots \sqcup I_{k} =\{1,\dots,n-1\}$ of the terms that we considered here, and we never used the restriction of $P$ to the simplex $a_1+\cdots+a_n=2g-3+n$, we have $P(a_1,\dots,a_{n-1},0)\equiv 0$.
\end{proof}

\begin{Corollary}The function $\tilde P(a_1,\dots,a_n)\coloneqq P(a_1,\dots,a_n)/\prod\limits_{i=1}^n a_i$ is a polynomial in $a_1,\dots, a_n$ of degree $2g-4+n$.
\end{Corollary}

So, we have a collection of symmetric polynomials of quite small degree (that is, smaller than what one expects trying to construct such non-trivial polynomials using the Lagrange interpolation, for instance) vanishing in all integer points of the certain simplices:
\begin{itemize}\itemsep=0pt
\item $P(a_1,\dots,a_n)$ is a polynomial of degree $2g-4+2n$ that vanishes at all integer points of the simplex $a_1,\dots,a_n\geq 0$, $a_1+\cdots+a_n=2g-3+n$.
\item $R(a_1,\dots,a_n)$ is a polynomial of degree $2g-2+n$ that vanishes at all integer points of the simplex $a_1,\dots,a_n\geq 0$, $a_1+\cdots+a_n=2g-3+n$.
\item $\tilde P(a_1+1,\dots,a_n+1)$ is a polynomial of degree $2g-4+n$ that vanishes at all integer points of the simplex $a_1,\dots,a_n\geq 0$, $a_1+\cdots+a_n=2g-3$.
\end{itemize}

\subsection[Combinatorial reduction of the identity for $a_i=1$]{Combinatorial reduction of the identity for $\boldsymbol{a_i=1}$}
In this section we give a combinatorial reduction of the identity~\eqref{eq:keyeq} in the case one of the arguments $a_i$ is equal to $1$.

\begin{Proposition}
For any $g\geq 2$ and $n\geq 1$, for any $a_1,\dots,a_n\in\mathbb{Z}_{\geq 0}$, we have
\begin{gather} \label{eq:reduction1}
 P(a_1,\dots,a_n,1) - P(a_1,\dots,a_n,0)= P(a_1,\dots,a_n) \cdot \left(4\sum_{i=1}^n a_i -8g+10-2n\right).
\end{gather}
\end{Proposition}

\begin{proof}Let $\{n+1\}\sqcup J\subset \{1,\dots,n+1\}$. Consider the corresponding factor in a summand in $P(a_1,\dots,a_n,1)$ assuming $I_j\coloneqq\{n+1\}\sqcup J$, for some $j$, and denoting the corresponding index~$d_j$ by~$d$. We have the following decomposition:
\begin{gather}
\binom{2(a_{[J]}+1)+1}{2d} \frac{(2d-1)!!}{(2d-1-2|J|)!!} = \binom{2(a_{[J]}+0)+1}{2d} \frac{(2d-1)!!}{(2d-1-2|J|)!!}
\notag\\
\qquad{} + \binom{2a_{[J]}+1}{2\tilde d} \frac{(2\tilde d-1)!!}{(2\tilde d+1-2|J|)!!} \cdot \big(4a_{[J]}+3-2\tilde d\big),\label{eq:decompositiona1}
\end{gather}
where $\tilde d = d-1$. The first term on the right hand side is equal to the corresponding factor in the same summand in $P(a_1,\dots,a_n,0)$.
The second term gives a summand in $P(a_1,\dots,a_n)$ with a coefficient.

There are $(k+1)$ ways to obtain the summand
\begin{gather*}
(-1)^k(2g-3+k)!\cdot \prod_{j=1}^k \binom{2a_{[I_j]}+1}{2d_j} \frac{(2d_j-1)!!}{(2d_j+1-2|I_j|)!!}
\end{gather*}
in $P(a_1,\dots,a_n)$ (here $I_1\sqcup\cdots\sqcup I_k = \llbracket n \rrbracket$, $d_1+\cdots+d_k = g-2+n$, and we omit the factor $1/k!$ that controls the permutations of the sets $I_1,\dots, I_k$) from the second term of the decomposition~\eqref{eq:decompositiona1}: either $J=I_j$, $j=1,\dots, k$, or $J=\varnothing$. In the latter case, the extra coefficient that we get is equal to $-3(2g-2+k)$. Thus, the total coefficient of this summand is equal to
\begin{gather*}
\sum_{j=1}^k (4a_{[J]}+3-2d_j) -3(2g-2+k) = 4\sum_{i=1}^n a_i -8g+10-2n,
\end{gather*}
which does not depend on the choice of $I_1\sqcup\cdots\sqcup I_k = \llbracket n \rrbracket$ and $d_1+\cdots+d_k = g-2+n$. This implies equation~\eqref{eq:reduction1}.
\end{proof}

If we restrict equation~\eqref{eq:reduction1} to the simplex $a_1+\cdots+a_n=2g-3+n$ and use that $P(a_1,\dots,a_n,0)\equiv 0$
(without any assumptions on $a_1,\dots,a_n$), we obtain the following corollary:

\begin{Corollary} For any $g\geq 2$ and $n\geq 1$, for any $a_1,\dots,a_n\in\mathbb{Z}_{\geq 0}$, $a_1+\cdots+a_n=2g-3+n$, we have
\begin{gather*}
P(a_1,\dots,a_n,1) = (2n-2) P(a_1,\dots,a_n).
\end{gather*}
In particular, $P(2g-2,1)=0$, and for $n\geq 2$ the vanishing of $P(a_1,\dots,a_n)$ implies the vanishing of $P(a_1,\dots,a_n,1)$.
\end{Corollary}

\subsection{A conjectural refinement of the identity}
In this section we formulate a conjectural refinement of the identity~\eqref{eq:keyeq}, which gives a natural strategy for its combinatorial proof. In particular, it allows to prove it for $n\leq 5$ for all $g$.

We replace each factor $\binom{2a_{[I]}+1}{2d}$ in each summand of the identity by the sum $\binom{2a_{[I]}}{2d}+\binom{2a_{[I]}}{2d-1}$. Then we collect all terms with the fixed number of factors where we have chosen to decrease $2d$ to $2d-1$. Let us define $P_{n,t}$ as
\begin{gather*}
\sum_{k=1}^n \frac{(-1)^k(2g-3+k)!}{k!}\\
\qquad{}\times
\sum_{\substack{I_1\sqcup\cdots\sqcup I_k = \llbracket n \rrbracket\\ I_j\neq \varnothing, \, \forall\, j\in \llbracket k \rrbracket}}
\sum_{\substack{d_1,\dots,d_k \in\mathbb{Z}_{\geq 0} \\ d_1+\cdots+d_k=g-2+n}}
\sum_{\substack{A\subset \llbracket k \rrbracket \\ |A|=n-t}}
\prod_{j=1}^k \binom{2a_{[I_j]}}{2d_j-\delta_{j\in A}} \frac{(2d_j-1)!!}{(2d_j+1-2|I_j|)!!}
\end{gather*}
for $t=0,\dots,n$. Here $\delta_{j\in A}$ is equal to $1$ for $j\in A$ and to $0$ otherwise.
For instance,
\begin{gather*}
P_{n,0} =(-1)^n(2g-3+n)!
\sum_{\substack{o_1,\dots,o_n \in(2\mathbb{Z}+1)_{>0} \\ o_1+\cdots+o_n=2g-4+n}}
\prod_{j=1}^n \binom{2a_j}{o_j}, \\
P_{n,1} =(-1)^n(2g-3+n)! \sum_{i=1}^n
\sum_{\substack{o_1,\dots,
\widehat{o_i}
,\dots o_n \in(2\mathbb{Z}+1)_{>0}
\\
e_i\in (2\mathbb{Z})_{\geq 0}
\\
o_1+\cdots\widehat{o_i} \cdots+o_n+ e_i=2g-3+n}}
\prod_{\substack{j=1\\ j\not=i}}^n \binom{2a_j}{o_j} \binom {2a_i}{e_i} \\
\hphantom{P_{n,0} =}{}
+(-1)^{n-1}(2g-4+n)! \sum_{\substack{i,\ell=1\\ i<\ell}}^n
\sum_{\substack{o_1,\dots,
\widehat{o_i}, \dots, \widehat{o_\ell}
,\dots o_n \in(2\mathbb{Z}+1)_{>0}
\\
o_{i\ell}\in (2\mathbb{Z}+1)_{> 0}
\\
o_1+\cdots\widehat{o_i},\widehat{o_\ell} \cdots+o_n+ o_{i\ell}=2g-3+n}}
\prod_{\substack{j=1\\ j\not=i,\ell}}^n \binom{2a_j}{o_j} \binom {2a_i+2a_\ell}{o_{i\ell}}\cdot o_{i\ell},
\\
 P_{n,n} =
 \sum_{k=1}^n \frac{(-1)^k(2g-3+k)!}{k!}
 \sum_{\substack{I_1\sqcup\cdots\sqcup I_k \\ = \{1,\dots,n\}}}
 \sum_{\substack{e_1,\dots,e_k \in(2\mathbb{Z})_{\geq 0} \\ e_1+\cdots+e_k=2g-4+2n}}
 \prod_{j=1}^k \binom{2a_{[I_j]}}{e_j} \frac{(e_j-1)!!}{(e_j+1-2|I_j|)!!}.
\end{gather*}
Here we use the notation $o_\bullet$ (resp., $e_\bullet$) to stress that these are odd (resp., even) non-negative numbers, and $\widehat{o_i}$ means that this particular index is skipped.

Denote by $A_n$ the sum $(-1)^n(2g-4+n)! \sum_{\substack{o_1,\dots,o_n \in(2\mathbb{Z}+1)_{>0} \\ o_1+\cdots+o_n=2g-4+n}}
\prod_{j=1}^n \binom{2a_j}{o_j}$.

\begin{Conjecture}\label{conj:comb}
For any $n\geq 2$ and $t=0,\dots,n$, $a_1,\dots,a_n\in\mathbb{Z}_{\geq 0}$, $a_1+\cdots+a_n=2g-3+n$, we have
\begin{gather} \label{eq:combconj}
P_{n,t} = (-1)^t \left[\left(\binom{n-1}{t}-\binom{n-1}{t-1}\right)(2g-3+n+t) +2(t-1)\binom{n-1}{t-1}\right] A_n.
\end{gather}
\end{Conjecture}
Observe that the right hand side is equal to
\begin{gather*}
(-1)^t \left[\binom{n-1}{t} (2g-2+n+(t-1))-\binom{n-1}{t-1}(2g-2+n-(t-1)) \right] A_n.
\end{gather*}

\begin{Remark} \label{rmk:conjimpliesid}
This conjecture does not follow from identity~\eqref{eq:keyeq}, so the equivalence of Faber's conjecture and identity~\eqref{eq:keyeq} does not prove equation~\eqref{eq:combconj}. On the other hand, let us prove that equation~\eqref{eq:combconj} implies the main combinatorial identity~\eqref{eq:keyeq}. Indeed,
\begin{gather*}
\sum_{t=0}^n (-1)^t \left[\binom{n-1}{t} (2g-2+n+(t-1))-\binom{n-1}{t-1}(2g-2+n-(t-1))\right]\\
\qquad{} = (2g-3+n) \sum_{t=0}^n (-1)^t \binom{n-1}{t}+ (n-1) \sum_{t=0}^n (-1)^t \binom{n-2}{t-1}\\
\qquad\quad{} - (2g-2+n) \sum_{t=0}^n (-1)^t \binom{n-1}{t-1} + (n-1) \sum_{t=0}^n (-1)^t \binom{n-2}{t-2} =0.
\end{gather*}
Therefore, a combinatorial proof of Conjecture~\ref{conj:comb} would immediately give a new proof of Faber's conjecture.
\end{Remark}

\begin{Proposition}\label{prop:conjtrue}
Conjecture~{\rm \ref{conj:comb}} is true for $n\leq 5$, any $t$, and for $t=0,1,2,3$, any $n$.
\end{Proposition}
We prove this proposition in Appendix~\ref{sec:ProofSpecialCases}.

\begin{Remark} Note that surprisingly this proposition is also true for $n=1$, though in this case we have no identity~\eqref{eq:keyeq}. Indeed, for $n=1$ we have
\begin{gather*}
A_1 = -(2g-3)! \binom{4g-4}{2g-3}, \\
P_{1,0} = -(2g-2)! \binom{4g-4}{2g-3} = (2g-2)A_1, \\
P_{1,1} = -(2g-2)! \binom{4g-4}{2g-2} = (2g-1)A_1,
\end{gather*}
which matches exactly equation~\eqref{eq:combconj} for $n=1$ and $t=0,1$.
\end{Remark}

\subsection{An equivalent formulation of the conjecture}
In this section we reformulate Conjecture~\ref{conj:comb} via a $3$-term recursion in the $P_{n,t}$. Let $\tilde{P}_{n,t}$ be
\begin{gather*}
\tilde{P}_{n,t} \coloneqq (-1)^t \left[\left(\binom{n-1}{t}-\binom{n-1}{t-1}\right)(2g-3+n+t) +2(t-1)\binom{n-1}{t-1}\right] A_n.
\end{gather*}

\begin{Proposition} Let $n\geq 2$ and $a_1,\dots,a_n\in\mathbb{Z}_{\geq 0}$, $a_1+\cdots+a_n=2g-3+n$. The following three statements, are equivalent:
\begin{enumerate}\itemsep=0pt
\item[$i)$] Conjecture~{\rm \ref{conj:comb}} holds:
\begin{gather*}
P_{n,t} = \tilde{P}_{n,t} \qquad \text{for all $t = 0,1, \dotsc, n$}.
\end{gather*}
\item[$ii)$] The $P_{n,t}$ obey the following $3$-term recursion for all $ t = 0, \dotsc, n$:
\begin{gather*}
(t+1) P_{n, t+1} + (n - (t+1))P_{n , t} = \frac{(-1)^t (2g - 1)}{(2g - 3 + n)} \binom{n}{t} P_{n,0}.
\end{gather*}
\item[$iii)$] The following expression does not depend on $t$:
\begin{gather*}
t! (n-t)! (-1)^t \left[(t+1) P_{n, t+1} + (n - (t+1))P_{n , t} \right] .
\end{gather*}
\end{enumerate}
\begin{proof}
Let
$S_{g,n}(x) \coloneqq \sum_{t=0}^n P_{n,t} x^t$,
for which we already know the values
\begin{gather*}
S_{g,n}(0)= P_{n,0} ,\\
S_{g,n}(1) = \sum_{t=0}^n P_{n,t}\\
\hphantom{S_{g,n}(1)}{} =\sum_{k=1}^n \frac{(-1)^k(2g-3+k)!}{k!}
\!\!\!\!\sum_{\substack{I_1\sqcup\cdots\sqcup I_k = \llbracket n \rrbracket\\ I_j\neq \varnothing,\, \forall\, j\in \llbracket k \rrbracket}}
\sum_{\substack{d_1,\dots,d_k \in\mathbb{Z}_{\geq 0} \\ \sum_i d_i =g-2+n}}
\prod_{j=1}^k \binom{2a_{[I_j]}+1}{2d_j} \frac{(2d_j-1)!!}{(2d_j+1-2|I_j|)!!}.
\end{gather*}
Let us now compute the generating polynomial for the $\tilde{P}_{n,t}$
\begin{gather*}
\tilde{S}_{g,n}(x) \coloneqq \sum_{t=0}^n \tilde{P}_{n,t} x^t \\
\hphantom{\tilde{S}_{g,n}(x)}{} = \sum_{t=0}^n (-x)^t \left[\left(\binom{n - 1}{t}-\binom{n - 1}{t - 1}\right)(2g-3+n+t) +2(t - 1)\binom{n - 1}{t - 1}\right] A_n.
\end{gather*}
First substitute $A_n = P_{n,0}/(2g - 3 + n)$ and expand the whole expression in terms of the type
\begin{gather*}
\sum_{t=0}^m \binom{m}{t} p(t) (-x)^t x^a,
\end{gather*}
where $p(t)$ is a polynomial in $t$ and $a$ is an integer. For each such summand substitute $p(t)$ with $p\big(x \frac{{\rm d}}{{\rm d}x}\big)$, apply Newton binomial theorem to $\sum_{t=0}^m \binom{m}{t} (-x)^t = (1-x)^n$, and finally apply the operator $p\big(x \frac{{\rm d}}{{\rm d}x}\big)$ to the summand. Collecting the summands' resulting contributions together gives
\begin{gather}
\tilde{S}_{g,n}(x) = \frac{(1-x)^{n-1}}{(2g - 3 + n)}\left( (2g-1)(x+1) + (n - 2)\right)P_{n,0}.
\end{gather}
Observe that
$
S_{g,n}(0) = \tilde{S}_{g,n}(0),
$
and $\tilde{S}(x)$ satisfies the non-homogeneous first order ODE
\begin{align}
(2g - 3 + n )\left[ (1-x) f'(x) + ( n -1 ) f(x) \right] = (1-x)^n (2g - 1 ) P_{n,0},
\end{align}
which $S_{g,n}(x)$ also satisfies if and only if
\begin{gather}\label{eq:propPnt}
(t+1) P_{n, t+1} + (n - (t+1))P_{n , t} = \frac{(-1)^t (2g - 1)}{(2g - 3 + n)} \binom{n}{t} P_{n,0},
\end{gather}
for $t = 0, 1, \dots, n$. This proves the equivalence between $(i)$ and $(ii)$. Clearly $(ii)$ implies $(iii)$. Let us see that $(iii)$ also implies $(ii)$. Assuming $(iii)$, we can evaluate $(ii)$ at any $t$. Let us pick $t=0$ for simplicity. Then $(ii)$ reads
\begin{align*}
P_{n, 1} &=\frac{ (2g - 1)}{(2g - 3 + n)} P_{n,0} - (n-1)P_{n,0}=\frac{ (2g - 1) - (n - 1)(2g - 3 + n )}{(2g - 3 + n)} P_{n,0}\\
&=\frac{ (2-n)(2g -2 +n)}{(2g - 3 + n)} P_{n,0}\\
&= - \left[\left(\binom{n-1}{1}-\binom{n-1}{0}\right)(2g-2+n) +2 \cdot 0 \cdot \binom{n-1}{0}\right] A_n,
\end{align*}
which holds true from the case $t=1$ in Proposition~\ref{prop:conjtrue}. This concludes the proof of the proposition.
\end{proof}
\end{Proposition}

\appendix

\section{Proof of the main combinatorial identity for several cases}\label{sec:ProofSpecialCases}

In this section we prove Proposition~\ref{prop:conjtrue}. By Remark~\ref{rmk:conjimpliesid}, this implies a proof of the main combinatorial identity~\eqref{eq:keyeq} for $n\leq 5$ and $g\geq 2$. By Corollary~\ref{cor:identity}, this implies a new proof of Faber's conjecture for $n\leq 5$ and $g\geq 2$.

We would like to emphasize that we present in this proof a brief exposition of a quite intricate computation involving an interplay between binomials with even and odd denominators that we could not find deeply analyzed in the literature on combinatorics.

\begin{proof}[Proof of Proposition~\ref{prop:conjtrue}] The case $t=0$ is obvious from the definition. For the other cases, we perform quite intricate computations based on the following lemma.

\begin{Lemma}\label{lemma:comb}For any non-negative integers $a_1,\dots,a_n$, $a_1+\cdots+a_n=A$, and $t_1,\dots,t_n$, $t_1+\cdots+t_n=T$, and for an arbitrary vector of parities $(p_1,\dots,p_n)$, $p_i\in \mathbb{Z}_2$, we have
\begin{gather*}
\sum_{\substack{f_1+\cdots+f_n=B \\ \tilde f_i = p_i,\ i=1,\dots,n}} \prod_{i=1}^n \binom{2a_i}{f_i} (f_i)_{t_i}
=
\sum_{\substack{f_1+\cdots+f_n=2A-B+T \\ \tilde f_i = p_i+\tilde t_i,\ i=1,\dots,n}} \prod_{i=1}^n \binom{2a_i}{f_i} (f_i)_{t_i} .
\end{gather*}
Here by $\tilde f \in \mathbb{Z}_2$ we denote the parity of $f\in\mathbb{Z}$, and by $(f)_t$ we denote the Pochhammer symbol, $(f)_t\coloneqq f(f-1)\cdots (f+1-t)$.
\end{Lemma}

\begin{proof} It follows from the following identity
\begin{gather*}
\binom{2a}{f}(f)_t = \binom{2a-t}{f-t} (2a)_t = \binom{2a-t}{2a-f} (2a)_t = \binom{2a}{2a-f+t} (2a-f+t)_t.\tag*{\qed}
\end{gather*}\renewcommand{\qed}{}
\end{proof}

\begin{Example} If $a_1+\cdots+a_n=2g-3+n$, then
\begin{gather*}
\sum_{\substack{o_1,\dots,o_n \in(2\mathbb{Z}+1)_{>0} \\ o_1+\cdots+o_n=2g-4+n}}
\prod_{j=1}^n \binom{2a_j}{o_j}
=
\sum_{\substack{o_1,\dots,o_n \in(2\mathbb{Z}+1)_{>0} \\ o_1+\cdots+o_n=2g-2+n}}
\prod_{j=1}^n \binom{2a_j}{o_j}.
\end{gather*}
Thus we have an alternative definition of $A_n$ as the sum
\begin{gather*}
(-1)^n(2g-4+n)!
\sum_{\substack{o_1,\dots,o_n \in(2\mathbb{Z}+1)_{>0} \\ o_1+\cdots+o_n=2g-2+n}}
\prod_{j=1}^n \binom{2a_j}{o_j}.
\end{gather*}
\end{Example}

Below in all arguments we apply Lemma~\ref{lemma:comb} assuming the condition $a_1+\cdots+a_n=2g-3+n$. We will also make use several times of the Chu--Vandermonde identity $\sum_{k+\ell=n} \binom{r}{k}\binom{s}{\ell} = \binom {r+s}{n}$, which follows from the expansion of the identity $(1+x)^r(1+x)^s = (1+x)^{r+s}$.

\subsection[Case $t=1$]{Case $\boldsymbol{t=1}$}

We have
\begin{gather*}
P_{n,1} =(-1)^n(2g-3+n)! \sum_{i=1}^n
\sum_{\substack{o_1,\dots,
\widehat{o_i}
,\dots o_n \in(2\mathbb{Z}+1)_{>0}
\\
e_i\in (2\mathbb{Z})_{\geq 0}
\\
o_1+\cdots\widehat{o_i} \cdots+o_n+ e_i=2g-3+n}}
\prod_{\substack{j=1\\ j\not=i}}^n \binom{2a_j}{o_j} \binom {2a_i}{e_i} \\
\hphantom{P_{n,1} =}{}
+(-1)^{n-1}(2g-4+n)! \sum_{\substack{i,\ell=1\\ i<\ell}}^n
\sum_{\substack{o_1,\dots,
\widehat{o_i}, \dots, \widehat{o_\ell}
,\dots o_n \in(2\mathbb{Z}+1)_{>0}
\\
o_{i\ell}\in (2\mathbb{Z}+1)_{> 0}
\\
o_1+\cdots\widehat{o_i},\widehat{o_\ell} \cdots+o_n+ o_{i\ell}=2g-3+n}}
\prod_{\substack{j=1\\ j\not=i,\ell}}^n \binom{2a_j}{o_j} \binom {2a_i+2a_\ell}{o_{i\ell}}\cdot o_{i\ell} .
\end{gather*}

In the first term we can replace the factor $(2g-3+n)$ by the sum of the indices $o_1+\cdots\widehat{o_i} \cdots+o_n+ e_i$. In the second term we can apply the Chu-Vandermonde identity
\begin{gather*}
\binom {2a_i+2a_\ell}{o_{i\ell}}\cdot o_{i\ell}=\sum_{\substack{o \in(2\mathbb{Z}+1)_{\geq 0}
\\
e\in (2\mathbb{Z})_{> 0}
\\
o+e=o_{il}}} \left(\binom{2a_i}{o}\binom{2a_{\ell}}{e}+\binom{2a_i}{e}\binom{2a_{\ell}}{o}\right)(e+o).
\end{gather*}
On the right-hand side of this equation, we always have one even bottom argument in the binomial coefficients and all other bottom arguments are odd in both terms of this expression. The expression is totally symmetric with respect to the choice of the place of the even bottom argument. The coefficient in each summand of
\begin{gather*}\label{eq:t1An}
(-1)^n(2g-4+n)! \sum_{\substack{
o_2\dots,o_n \in(2\mathbb{Z}+1)_{>0}
\\
e_1\in (2\mathbb{Z})_{\geq 0}
\\
e_1+o_2+\cdots+o_n=2g-3+n}
}
\binom {2a_1}{e_1} \prod_{j=2}^n \binom{2a_j}{o_j}
\end{gather*}
(that is, we collect the terms where the even bottom argument is below $2a_1$) is equal to
\begin{gather*}\label{eq:t1Coeff}
\left(e_1+\sum_{i=2}^n o_i\right)-\sum_{i=2}^n (e_1+o_i) = -e_1\cdot \left(
\binom{n-1}{1} - \binom{n-1}{0}\right).
\end{gather*}
Applying Lemma~\ref{lemma:comb} to this term, we obtain
\begin{gather*}
(-1)^n(2g-4+n)! \!\!\! \sum_{\substack{
o_2\dots,o_n \in(2\mathbb{Z}+1)_{>0}
\\
e_1\in (2\mathbb{Z})_{\geq 0}
\\
e_1+o_2+\cdots+o_n=2g-3+n}
}\!\!
\binom {2a_1}{e_1} \prod_{j=2}^n \binom{2a_j}{o_j}
\cdot
(-1)\cdot e_1\cdot \left(
\binom{n-1}{1} - \binom{n-1}{0}
\right)
\\
\qquad{} = (-1)^n(2g-4+n)! \!\!\!
\sum_{\substack{o_1,\dots,o_n \in(2\mathbb{Z}+1)_{>0} \\ o_1+\cdots+o_n=2g-2+n}}
\prod_{j=1}^n \binom{2a_j}{o_j} \cdot o_1 \cdot (-1)\cdot \left(
\binom{n-1}{1} - \binom{n-1}{0}
\right).
\end{gather*}
Now, since the even bottom argument could be at any place, not only at the first one, we have to replace in the full computation of $P_{n,1}$ the factor $o_1$ above by $o_1+\cdots+o_n=2g-2+n$. Thus we have
\begin{gather*}
P_{n,1} = A_n\cdot (2g-2+n)\cdot (-1)\cdot \left(\binom{n-1}{1} - \binom{n-1}{0}\right),
\end{gather*}
which is exactly the desired result for $t=1$.

\subsection[Case $t=2$]{Case $\boldsymbol{t=2}$}
 Let us describe $P_{n,2}$. All terms there have a common factor of $(2g-5+n)!$. The sum of bottom arguments of all binomial coefficients is always equal to $S\coloneqq 2g-2+n$. Taking into account the total symmetry with respect to the permutations of $a_1,\dots,a_n$, we see that $(-1)^nP_{n,2}/(2g-5+n)!$ has
\begin{alignat*}{3}
&\binom{n}{3} \ \text{terms of the type} \quad && \binom{2a_1+2a_2+2a_3}{o_{123}} \prod_{i=4}^n \binom{2a_i}{o_i} \cdot o_{123}(o_{123}-2), & \\
&3\binom{n}{4} \ \text{terms of the type} \quad && \binom{2a_1+2a_2}{o_{12}} \binom{2a_3+2a_4}{o_{34}}\prod_{i=5}^n \binom{2a_i}{o_i} \cdot o_{12}o_{34}, & \\
&\binom{n}{2}\binom{n-2}{1}\ \text{terms of the type} \quad && -\binom{2a_1+2a_2}{o_{12}} \binom{2a_3}{e_{3}}\prod_{i=4}^n \binom{2a_i}{o_i} \cdot o_{12}(S-2), & \\
&\binom{n}{2}\ \text{terms of the type} \quad && -\binom{2a_1+2a_2}{e_{12}} \prod_{i=3}^n \binom{2a_i}{o_i} \cdot (e_{12}-1)(S-2), & \\
&\binom{n}{2}\ \text{terms of the type} \quad && \binom{2a_1}{e_{1}} \binom{2a_2}{e_{2}} \prod_{i=3}^n \binom{2a_i}{o_i} \cdot (S-1)(S-2).&
\end{alignat*}
For instance, in the first line we mean that we have the following sum of $\binom{n}{3}$ summands
\begin{gather*}
\sum_{i<j<k} \sum_{\substack{
o_{ijk} \in (2\mathbb{Z}+1)_{>0} \\
o_\ell \in (2\mathbb{Z}+1)_{>0},\ \ell\in \{1,\dots,n\}\setminus \{i,j,k\} \\
o_{ijk}+\sum_{\ell\in \{1,\dots,n\}\setminus \{i,j,k\} }o_\ell = 2g-2+n
}}
\binom{2a_i+2a_j+2a_k}{o_{ijk}} \prod_{\substack{\ell=1 \\ \ell\not=i,j,k}}^n \binom{2a_\ell}{o_\ell} \cdot o_{ijk}(o_{ijk}-2).
\end{gather*}
We assume that the parity of the bottom arguments denoted by $o$ (resp., $e$) is odd (resp., even).

Let us expand all binomial coefficients using the Chu--Vandermonde identity, that is, in such a way that we have exactly $n$ factors of the type $\binom{2a_i}{f_i}$, where we also keep track of the possible parity of the bottom arguments. For instance,
\begin{gather*}
\binom{2a_1+2a_2+2a_3}{o_{123}} = \sum_{e_1+e_2+o_3=o_{123}}\!\! \binom{2a_1}{e_1}\binom{2a_2}{e_2} \binom{2a_3}{o_3}
+ \sum_{e_1+o_2+e_3=o_{123}} \!\!\binom{2a_1}{e_1}\binom{2a_2}{o_2} \binom{2a_3}{e_3}\\
 \qquad{} + \sum_{o_1+e_2+e_3=o_{123}} \binom{2a_1}{o_1}\binom{2a_2}{e_2} \binom{2a_3}{e_3}
+ \sum_{o_1+o_2+o_3=o_{123}} \binom{2a_1}{o_1}\binom{2a_2}{o_2} \binom{2a_3}{o_3}.
\end{gather*}
We compute all the coefficients to obtain
\begin{alignat*}{3}
&\binom{n}{2} \ \text{terms of the type} \quad && \binom{2a_1}{e_{1}} \binom{2a_2}{e_{2}} \prod_{i=3}^n \binom{2a_i}{o_i}\cdot 2e_1e_2\binom{n-2}{2}, & \\
&\binom{n}{1} \ \text{terms of the type} \quad&& \prod_{i=1}^n \binom{2a_i}{o_i}\cdot o_1(o_1-1) \left(\binom{n-1}{2} -\binom{n-1}{1}\right), & \\
&\binom{n}{2} \ \text{terms of the type} \quad&& \prod_{i=1}^n \binom{2a_i}{o_i}\cdot 2o_1o_2 \left(\binom{n-2}{1} -\binom{n-1}{1}\right), & \\
&\binom{n}{1} \ \text{terms of the type} \quad&& \prod_{i=1}^n \binom{2a_i}{o_i}\cdot o_1 \left(-\binom{n-1}{2} +\binom{n-1}{1}+\binom{n}{2}\right), &\\
&1 \ \text{term of the type} \quad && -\prod_{i=1}^n \binom{2a_i}{o_i}\cdot 2\binom{n}{2}. &
\end{alignat*}
Applying Lemma~\ref{lemma:comb} to the sum of all terms in the first line, we obtain the same terms as in the third line, with the coefficient $2o_1o_2\binom{n-2}2$. This, together with all terms in the second line and the third line, gives us the term in the fifth line with the coefficient $\left(\binom{n-1}{2}-\binom{n-1}{1}\right) S(S-1)$. The sum of all terms in the fourth line gives us also the term in the fifth line with the coefficient $S\left(-\binom{n-1}{2} +\binom{n-1}{1}+\binom{n}{2}\right)$. The observation that
\begin{gather*}
\left(\binom{n-1}{2}-\binom{n-1}{1}\right) S(S-1)+S\left(-\binom{n-1}{2} +\binom{n-1}{1}+\binom{n}{2}\right)-2\binom{n}{2}\\
\qquad{} = (S-2)\cdot \left[\left(\binom{n-1}{2}-\binom{n-1}{1}\right)(S+1) +2\binom{n-1}{1}\right]
\end{gather*}
is exactly the product of $(2g-4+n)$ and the desired coefficient of $P_{n,2}/A_n$ completes the proof of this case.

\subsection[Case $t=3$]{Case $\boldsymbol{t=3}$}
 Let us describe $P_{n,3}$. All terms there have a common factor of $(2g-6+n)!$. The sum of bottom arguments of all binomial coefficients is always equal to $S\coloneqq2g-1+n$. Taking into account the total symmetry with respect to the permutations of $a_1,\dots,a_n$, we see that $(-1)^nP_{n,3}/(2g-6+n)!$ has terms of the following type:
\begin{gather*}
-\binom{2a_1+2a_2+2a_3+2a_4}{o_{1234}} \prod_{i=5}^n \binom{2a_i}{o_i} \cdot o_{1234}(o_{1234}-2)(o_{1234}-4), \\
-\binom{2a_1+2a_2+2a_3}{o_{123}} \binom{2a_4+2a_5}{o_{45}}\prod_{i=6}^n \binom{2a_i}{o_i} \cdot o_{123}(o_{123}-2)o_{45}, \\
-\binom{2a_1+2a_2}{o_{12}} \binom{2a_3+2a_4}{o_{34}}\binom{2a_5+2a_6}{o_{56}}\prod_{i=7}^n \binom{2a_i}{o_i} \cdot o_{12}o_{34}o_{56}, \\
\binom{2a_1+2a_2+2a_3}{o_{123}} \binom{2a_4}{e_4} \prod_{i=5}^n \binom{2a_i}{o_i} \cdot o_{123}(o_{1234}-2)(S-4), \\
\binom{2a_1+2a_2+2a_3}{e_{123}} \prod_{i=4}^n \binom{2a_i}{o_i} \cdot (e_{123}-1)(e_{123}-3)(S-4), \\
\binom{2a_1+2a_2}{o_{12}} \binom{2a_3+2a_4}{o_{34}} \binom{2a_5}{e_5} \prod_{i=6}^n \binom{2a_i}{o_i} \cdot o_{12}o_{34} (S-4) ,\\
\binom{2a_1+2a_2}{e_{12}} \binom{2a_3+2a_4}{o_{34}} \prod_{i=5}^n \binom{2a_i}{o_i} \cdot (e_{12}-1)o_{34} (S-4) ,\\
-\binom{2a_1+2a_2}{o_{12}} \binom{2a_3}{e_3} \binom{2a_4}{e_4} \prod_{i=5}^n \binom{2a_i}{o_i} \cdot o_{12} (S-3) (S-4) ,\\
-\binom{2a_1+2a_2}{e_{12}} \binom{2a_3}{e_3} \prod_{i=4}^n \binom{2a_i}{o_i} \cdot (e_{12}-1) (S-3) (S-4) ,\\
\binom{2a_1}{e_1}\binom{2a_2}{e_2} \binom{2a_3}{e_3} \prod_{i=4}^n \binom{2a_i}{o_i} \cdot (S-2)(S-3)(S-4).
\end{gather*}

Let us expand all binomial coefficients using the Chu--Vandermonde identity, that is, in such a way that we have exactly $n$ factors of the type $\binom{2a_i}{f_i}$, where we also keep track of the possible parity of the bottom arguments. Computing the coefficients, we obtain terms of the type $\binom{2a_1}{e_1}\binom{2a_2}{e_2} \binom{2a_3}{e_3} \prod_{i=4}^n \binom{2a_i}{o_i} $ and $\binom{2a_1}{e_1}\prod_{i=2}^n \binom{2a_i}{o_i}$ with some complicated coefficients that we want to collect in several disjoint groups.

\subsubsection{First group of terms} Denote $-\left( \binom{n-1}{3} - \binom{n-1}{2} \right)$ by $C_1$. With this coefficient we have terms of the following type:
\begin{gather*}
 \binom{2a_1}{e_1}\prod_{i=2}^n \binom{2a_i}{o_i} \cdot e_1(e_1-1)(e_1-2) \cdot C_1, \\
 \binom{2a_1}{e_1}\binom{2a_2}{e_2} \binom{2a_3}{e_3} \prod_{i=4}^n \binom{2a_i}{o_i} \cdot e_1e_2e_3 \cdot 6C_1, \\
 \binom{2a_1}{e_1}\binom{2a_2}{e_2}\prod_{i=3}^n \binom{2a_i}{o_i} \cdot e_1o_2(o_2-1) \cdot 3C_1.
\end{gather*}
Applying Lemma~\ref{lemma:comb} to these terms, we obtain
\begin{gather*}
\sum_{\substack{o_1+\cdots+o_n \\ = 2g-2+n}} \prod_{i=1}^n \binom{2a_i}{o_i} \cdot (2g-2+n)(2g-3+n)(2g-4+n) C_1.
\end{gather*}

\subsubsection{Second group of terms} Denote $-4 \binom{n-1}{2}$ by $C_2$. With this coefficient we have terms of the following type:
\begin{gather*}
 \binom{2a_1}{e_1}\prod_{i=2}^n \binom{2a_i}{o_i} \cdot e_1 \cdot e_1C_2, \\
 \binom{2a_1}{e_1}\binom{2a_2}{o_2} \prod_{i=3}^n \binom{2a_i}{o_i} \cdot e_1 \cdot o_2C_2, \\
 \binom{2a_1}{e_1}\prod_{i=2}^n \binom{2a_i}{o_i} \cdot e_1 \cdot (-4)C_2.
\end{gather*}
We collect these terms into $\binom{2a_1}{e_1}\prod_{i=2}^n \binom{2a_i}{o_i} \cdot e_1 \cdot (2g-5+n)C_2$. Applying Lemma~\ref{lemma:comb} to all these terms, we obtain
\begin{gather*}
\sum_{\substack{o_1+\cdots+o_n \\ = 2g-4+n}} \prod_{i=1}^n \binom{2a_i}{o_i} \cdot (o_1+\cdots+o_n)(2g-5+n)C_2\\
\qquad{} = \sum_{\substack{o_1+\cdots+o_n \\ = 2g-4+n}} \prod_{i=1}^n \binom{2a_i}{o_i} \cdot (2g-4+n)(2g-5+n)C_2.
\end{gather*}
Applying Lemma~\ref{lemma:comb} again, we obtain
\begin{gather*}
\sum_{\substack{o_1+\cdots+o_n \\ = 2g-2+n}} \prod_{i=1}^n \binom{2a_i}{o_i} \cdot (2g-4+n)(2g-5+n)C_2.
\end{gather*}

\subsubsection{Third group of terms} Denote $6\left( \binom{n-1}{3} - \binom{n-1}{2} \right)$ by $C_3$. With this coefficient we have terms of the following type:
\begin{align*}
& \binom{2a_1}{e_1}\prod_{i=2}^n \binom{2a_i}{o_i} \cdot e_1 \cdot C_3.
\end{align*}
Applying Lemma~\ref{lemma:comb} to all these terms, we obtain
\begin{gather*}
\sum_{\substack{o_1+\cdots+o_n \\ = 2g-4+n}} \prod_{i=1}^n \binom{2a_i}{o_i} \cdot (o_1+\cdots+o_n)C_3
= \sum_{\substack{o_1+\cdots+o_n \\ = 2g-4+n}} \prod_{i=1}^n \binom{2a_i}{o_i} \cdot (2g-4+n)C_3.
\end{gather*}
Applying Lemma~\ref{lemma:comb} again, we obtain
\begin{gather*}
\sum_{\substack{o_1+\cdots+o_n \\ = 2g-2+n}} \prod_{i=1}^n \binom{2a_i}{o_i} \cdot (2g-4+n)C_3.
\end{gather*}

\subsubsection{Fourth group of terms} Denote $-2\binom{n-2}{1}$ by $C_4$. With this coefficient we have terms of the following type:
\begin{gather*}
\binom{2a_1}{e_1}\binom{2a_2}{e_2}\binom{2a_3}{e_3} \prod_{i=4}^n \binom{2a_i}{o_i} \cdot 3e_1e_2e_3 C_4, \\
\binom{2a_1}{e_1}\prod_{i=2}^n \binom{2a_i}{o_i} \cdot e_1o_2(o_2-1)C_4, \\
-\binom{2a_1}{e_1}\prod_{i=2}^n \binom{2a_i}{o_i} \cdot e_1(e_1-1)o_2C_4, \\
-\binom{2a_1}{e_1}\prod_{i=2}^n \binom{2a_i}{o_i} \cdot e_1o_2o_3C_4.
\end{gather*}
Applying Lemma~\ref{lemma:comb} to the first two lines, we obtain
\begin{gather*}
\sum_{\substack{o_1+\cdots+o_n \\ = 2g-2+n}} \prod_{i=1}^n \binom{2a_i}{o_i} \cdot \bigg(\sum_{i<j} o_io_j\bigg) \cdot (o_1+\cdots+o_n-2)C_4\\
\qquad{} = \sum_{\substack{o_1+\cdots+o_n \\ = 2g-2+n}} \prod_{i=1}^n \binom{2a_i}{o_i} \cdot \bigg(\sum_{i<j} o_io_j\bigg) \cdot (2g-4+n)C_4.
\end{gather*}
Applying Lemma~\ref{lemma:comb} to the last two lines, we obtain
\begin{gather*}
-\sum_{i<j} \sum_{\substack{e_i+e_j+\sum_{\ell\in\{1,\dots,n\}\setminus\{i,j\}}o_\ell \\ = 2g-2+n}}
\binom{2a_i}{e_i}\binom{2a_j}{e_j} \prod_{\substack{\ell=1\\ \ell\not=i,j}}^n \binom{2a_i}{o_i} \cdot e_ie_j \cdot \left(e_i+e_j+\!\!\! \sum_{\ell\in\{1,\dots,n\}\setminus\{i,j\}}\!\!\! o_\ell\right)C_4\\
\qquad{} = -\sum_{i<j} \sum_{\substack{e_i+e_j+\sum_{\ell\in\{1,\dots,n\}\setminus\{i,j\}}o_\ell \\ = 2g-2+n}}
\binom{2a_i}{e_i}\binom{2a_j}{e_j} \prod_{\substack{\ell=1\\ \ell\not=i,j}}^n \binom{2a_i}{o_i} \cdot e_ie_j \cdot (2g-4+n)C_4.
\end{gather*}
It follows from Lemma~\ref{lemma:comb} that
\begin{gather*}
\sum_{\substack{o_1+\cdots+o_n \\ = 2g-2+n}} \prod_{i=1}^n \binom{2a_i}{o_i} \cdot \left(\sum_{i<j} o_io_j\right)
-\sum_{i<j} \sum_{\substack{e_i+e_j\\ +\sum_{\ell\in\{1,\dots,n\}\setminus\{i,j\}}o_\ell \\ = 2g-2+n}} \!\!
\binom{2a_i}{e_i}\binom{2a_j}{e_j} \prod_{\substack{\ell=1\\ \ell\not=i,j}}^n \binom{2a_i}{o_i} \cdot e_ie_j =0.
\end{gather*}
Hence, the total sum of all terms with the coefficient $C_4$ is equal to $0$.

\subsubsection{Final computation} In order to complete the proof of the case $t=3$ it is sufficient to observe that
\begin{gather*}
 (2g-2+n)(2g-3+n)(2g-4+n)C_1+(2g-4+n)(2g-5+n)C_2+ (2g-4+n)C_3 \\
\qquad{} = -(2g-4+n)(2g-5+n) \left[\left(\binom{n-1}{3}-\binom{n-1}{2}\right)(2g+n) +4\binom{n-1}{2}\right].
\end{gather*}

\subsection[Case $n=t=4$]{Case $\boldsymbol{n=t=4}$} In this case $a_1+\cdots + a_4=2g+1$. We have the following formula for $P_{4,4}$:
\begin{gather*}
\frac{P_{4,4}}{(2g-1)!} =
\sum_{k=1}^4 \frac{(-1)^k(2g-3+k)!}{k!(2g-1)!}
\sum_{\substack{I_1\sqcup\cdots\sqcup I_k \\ = \{1,\dots,4\}}}
\sum_{\substack{e_1,\dots,e_k \in(2\mathbb{Z})_{\geq 0} \\ e_1+\cdots+e_k=2g+4}}
\prod_{j=1}^k \binom{2a_{[I_j]}}{e_j} \frac{(e_j-1)!!}{(e_j+1-2|I_j|)!!}.
\end{gather*}
Note that if $k=1$, then $(2g-3+k)!=(2g-2)!$. But then this term looks like
\begin{gather*}
\binom{2a_1+2a_2+2a_3+2a_4}{2g+4}(2g+3)(2g+1)(2g-1),
\end{gather*}
and the last factor here still allows us to extract the common coefficient of $(2g-1)!$. With that remark we see that every term in the expression for $\frac{P_{4,4}}{(2g-1)!}$ above is multiplied by a quadratic polynomial in $e_1,\dots,e_k$.

Applying the Chu--Vandermonde identity in the same way as in the previous cases, we obtain terms of the following type:
\begin{gather*}
 \prod_{i=1}^4 \binom{2a_i}{e_i} \cdot 0,
\qquad -\binom{2a_1}{e_1}\binom{2a_2}{e_2} \binom{2a_3}{o_3} \binom{2a_4}{o_3} \cdot 2e_1e_2,
\qquad -\prod_{i=1}^4 \binom{2a_i}{o_i} \cdot o_1(o_1-1).
\end{gather*}
Applying Lemma~\ref{lemma:comb} to all these terms, we obtain
\begin{align*}
& -\sum_{\substack{o_1+o_2+o_3 \\ +o_4= 2g}} \prod_{i=1}^4 \binom{2a_i}{o_i} \cdot (o_1+o_2+o_3+o_4)(o_1+o_2+o_3+o_4-1)
= -(2g-1)\cdot \frac{A_4}{(2g-1)!},
\end{align*}
which confirms this case of the proposition.

\subsection[Case $n=t=5$]{Case $\boldsymbol{n=t=5}$} In this case $a_1+\cdots + a_5=2g+2$. We have the following formula for $P_{5,5}$:
\begin{gather*}
\frac{P_{5,5}}{(2g-1)!} =
\sum_{k=1}^5 \frac{(-1)^k(2g-3+k)!}{k!(2g-1)!}
\sum_{\substack{I_1\sqcup\cdots\sqcup I_k \\ = \{1,\dots,5\}}}
\sum_{\substack{e_1,\dots,e_k \in(2\mathbb{Z})_{\geq 0} \\ e_1+\cdots+e_k=2g+6}}
\prod_{j=1}^k \binom{2a_{[I_j]}}{e_j} \frac{(e_j-1)!!}{(e_j+1-2|I_j|)!!}.
\end{gather*}
Note that if $k=1$, then $(2g-3+k)!=(2g-2)!$. But then this term looks like
\begin{gather*}
\binom{2a_1+2a_2+2a_3+2a_4+2a_5}{2g+6}(2g+5)(2g+3)(2g+1)(2g-1),
\end{gather*}
and the last factor here still allows us to extract the common coefficient of $(2g-1)!$. With that remark we see that every term in the expression for $\frac{P_{5,5}}{(2g-1)!}$ above is multiplied by a cubic polynomial in $e_1,\dots,e_k$.

Applying the Chu--Vandermonde identity in the same way as in the previous cases, we obtain terms of the following type:
\begin{gather*}
 -\binom{2a_1}{e_1}\binom{2a_2}{e_2} \binom{2a_3}{e_3} \binom{2a_4}{o_3} \binom{2a_5}{o_5} \cdot 6e_1e_2e_3, \\
 -\binom{2a_1}{e_1} \binom{2a_2}{o_2} \prod_{i=3}^5 \binom{2a_i}{o_i} \cdot 3e_1o_2(o_2-1), \\
 -\binom{2a_1}{e_1} \prod_{i=2}^5 \binom{2a_i}{o_i} \cdot e_1(e_1-1)(e_1-2).
\end{gather*}
Applying Lemma~\ref{lemma:comb} to all these terms, we obtain
\begin{align*}
& -\sum_{\substack{o_1+o_2+o_3 \\+o_4+o_5 = 2g+1}} \prod_{i=1}^5 \binom{2a_i}{o_i} \cdot \left(\sum_{i=1}^5 o_i\right)\left(\sum_{i=1}^5 o_i-1\right)\left(\sum_{i=1}^5 o_i-2\right)
= (2g-1)\cdot \frac{A_5}{(2g-1)!},
\end{align*}
which confirms this case of the proposition.

\subsection[Case $n=5$, $t=4$]{Case $\boldsymbol{n=5}$, $\boldsymbol{t=4}$} In this case $a_1+\cdots + a_5=2g+2$. We have the following formula for $P_{5,4}$:
\begin{gather*}
\frac{P_{5,4}}{(2g-1)!} =
\sum_{k=1}^5 \frac{(-1)^k(2g-3+k)!}{k!(2g-1)!}\\
\hphantom{\frac{P_{5,4}}{(2g-1)!} =}{}\times
\sum_{\substack{I_1\sqcup\cdots\sqcup I_k \\ = \{1,\dots,5\}}}
\sum_{\substack{e_1,\dots,e_k \in(2\mathbb{Z})_{\geq 0} \\ e_1+\cdots+e_k=2g+6}} \sum_{\ell=1}^k
\prod_{j=1}^k \binom{2a_{[I_j]}}{e_j-\delta_{\ell j}} \frac{(e_j-1)!!}{(e_j+1-2|I_j|)!!}.
\end{gather*}
Here we can divide by $(2g-1)!$ for the same reason as in the case $n=t=5$, and after that we can consider the coefficient of every term in this expression to be a cubic polynomial in $e_j-\delta_{\ell j}$.

We apply the Chu--Vandermonde identity in the same way as in the previous cases, and we obtain two groups of terms (the sum of the bottom arguments in the binomial coefficients in these terms is equal to $2g+5$).

The first group of terms consists of
\begin{alignat*}{3}
& 20 \ \text{terms of the type} \quad && \binom{2a_1}{e_1}\binom{2a_2}{e_2} \prod_{i=3}^5 \binom{2a_i}{o_i} \cdot (-6)e_1^2e_2, & \\
& 10 \ \text{terms of the type} \quad && \binom{2a_1}{e_1}\binom{2a_2}{e_2} \prod_{i=3}^5 \binom{2a_i}{o_i} \cdot 42 e_1e_2, & \\
& 30 \ \text{terms of the type} \quad && \binom{2a_1}{e_1} \binom{2a_2}{e_2} \binom{2a_3}{o_3} \prod_{i=4}^5 \binom{2a_i}{o_i} \cdot (-6) e_1e_2o_1.&
\end{alignat*}
Taking into account that $e_1+e_2+o_3+o_4+o_5=2g+5$, we see that the sum of all this terms is equal to
\begin{gather}\label{eq:case54-1}
10 \ \text{terms of the type} \quad \binom{2a_1}{e_1}\binom{2a_2}{e_2} \prod_{i=3}^5 \binom{2a_i}{o_i} \cdot (-6)e_1e_2(2g-2).
\end{gather}

The second group of terms consists of
\begin{alignat*}{3}
& 5 \ \text{terms of the type}\quad && \binom{2a_1}{o_1} \prod_{i=2}^5 \binom{2a_i}{o_i} \cdot (-3)o_1(o_1-1)(o_1-7), & \\
& 20 \ \text{terms of the type} \quad && \binom{2a_1}{o_1} \binom{2a_2}{o_2}\prod_{i=3}^5 \binom{2a_i}{o_i} \cdot (-3)o_1(o_1-1)o_2.&
\end{alignat*}
Taking into account that $o_1+o_2+o_3+o_4+o_5=2g+5$, we see that the sum of all this terms is equal to
\begin{gather}\label{eq:case54-2}
5 \ \text{terms of the type} \quad \binom{2a_1}{o_1} \prod_{i=2}^5 \binom{2a_i}{o_i} \cdot (-3)o_1(o_1-1)(2g-2).
\end{gather}

We apply Lemma~\ref{lemma:comb} to \eqref{eq:case54-1} and \eqref{eq:case54-2}, and this gives us
\begin{align*}
& \sum_{\substack{o_1+o_2+o_3 \\+o_4+o_5 = 2g+1}} \prod_{i=1}^5 \binom{2a_i}{o_i} \cdot (-3)\left(\sum_{i=1}^5 o_i\right)\left(\sum_{i=1}^5 o_i-1\right)
= -3(2g-2)\cdot \frac{A_5}{(2g-1)!},
\end{align*}
which confirms the proposition in this case. This concludes the proof of the proposition.
\end{proof}

\subsection*{Acknowledgments}
\vspace{-1mm}

We thank A.~Buryak, J.~Schmitt and D.~Zagier for useful comments on the first version of the paper. R.K., D.L., and S.S.\ were supported by the Netherlands Organization for Scientific Research. D.L.~was also supported by the Max Planck Gesellschaft. E.G.-F.~was supported by the Max Planck Gesellschaft and by the Labex Mathematics Hadamard. She is also grateful for the research stay at the University of Amsterdam, which made the beginning of this work possible. We thank the anonymous referees for many useful remarks.

\vspace{-3.2mm}

\pdfbookmark[1]{References}{ref}
\LastPageEnding

\end{document}